\documentclass[11pt, oneside]{amsart}

\usepackage[T1]{fontenc}
\usepackage{geometry}
\usepackage{amsmath}
\usepackage{mathrsfs}
\usepackage{mathtools}
\usepackage{tikz-cd}
\usepackage{amssymb,amsfonts,latexsym}
\usepackage{todonotes}
\usepackage{enumitem}
\usepackage{stmaryrd}
\usepackage{multirow}
\usepackage{scalerel}
\usepackage{longtable}

\usetikzlibrary{bbox}
\usepackage{verbatim}
\usepackage[all]{xy}

\newtheorem{thm}{Theorem}[section]
\newtheorem{cor}[thm]{Corollary}
\newtheorem{lem}[thm]{Lemma}
\newtheorem{prop}[thm]{Proposition}
\newtheorem{fact}[thm]{Fact}

\newtheorem{nota}[thm]{Notation}

\theoremstyle{definition}

\newtheorem{manualtheoreminner}{Tool}
\newenvironment{manualtheorem}[1]{%
  \renewcommand\themanualtheoreminner{#1}%
  \manualtheoreminner
}{\endmanualtheoreminner}
\theoremstyle{remark}
\newtheorem{rem}[thm]{Remark}
\newtheorem{remi}[thm]{Reminder}

\numberwithin{equation}{section}

\usepackage{chngcntr}
\counterwithin{table}{section}

\usepackage{thmtools, thm-restate}

\newcommand{\eps}{\varepsilon}

\newcommand{\fS}{\mathfrak{S}}
\newcommand{\fA}{\mathfrak{A}}

\DeclareMathOperator\charak{char}

\newcommand{\ra}{\rightarrow}

\DeclareMathOperator\Gal{Gal}

\newcommand\oline[1] {{\overline{#1}}}
\newcommand\Spec{{\operatorname{Spec}}}
\newcommand\Br{{\operatorname{Br}}}
\newcommand\Brun{{\operatorname{Br_{un}}}}
\newcommand\out{{\operatorname{Out}}}
\newcommand\disc{{\operatorname{disc}}}
\newcommand\den{{\operatorname{den}}}
\newcommand\et{{\operatorname{\acute{e}t}}}
\newcommand\ab{{\operatorname{ab}}}
\newcommand\sh{{\operatorname{sh}}}
\newcommand\Bij{{\operatorname{Bij}}}
\newcommand\cd{{\operatorname{cd}}}
\newcommand\Hom{{\operatorname{Hom}}}

\newcommand\Sym{{\operatorname{Sym}}}

\newcommand\SchSL{{\operatorname{SL}}}
\newcommand\SchGL{{\operatorname{GL}}}

\newcommand\Aut{{\operatorname{Aut}}}

\newcommand\Sec{{\operatorname{Sec}}}
\newcommand\PGL{{\operatorname{PGL}}}

\renewcommand\Im{{\operatorname{Im}}}
\newcommand\pr{{\operatorname{pr}}}
\newcommand\coh{{\operatorname{H}}}
\newcommand\coc{{\operatorname{Z}}}

\newcommand\Pic{{\operatorname{Pic}}}
\newcommand\Ind{{\operatorname{Ind}}}
\newcommand\Res{{\operatorname{Res}}}
\newcommand\cont{{\operatorname{cont}}}

\newcommand\GL[1][n] {{\operatorname{GL}_{#1}}}

\DeclareMathOperator\PSL{PSL}

\DeclareMathOperator\He{He}

\usepackage[colorlinks,pagebackref,pdftex, bookmarks=false]{hyperref}

\makeatletter
\def\blfootnote{\gdef\@thefnmark{}\@footnotetext}
\makeatother

\newcommand\nnfootnote[1]{%
  \begin{NoHyper}
  \renewcommand\thefootnote{}\footnote{#1}%
  \addtocounter{footnote}{-1}%
  \end{NoHyper}
}

\begin{document}

\nnfootnote{\textit{Date}: \today}

\title[Homogenous spaces with non-solvable stabilisers] 
{The Grunwald problem and homogeneous spaces with non-solvable stabilisers}

\author{Elyes Boughattas}
\address{Department of Mathematical Sciences, University of Bath - Claverton Down, Bath, BA2~7AY, United Kingdom}
\email{eb2751@bath.ac.uk}
\def\technion{Department of Mathematics, Technion - Israel Institute of Technology, Haifa 32000, Israel}
\author{Danny Neftin}
\address{\technion}
\email{dneftin@technion.ac.il}

\begin{abstract}
We give an affirmative answer to the Grunwald problem for new families of non-solvable finite groups $G$, away from the set of primes dividing $|G|$. Furthermore, we show that such $G$ verify the condition (BM), that is, the Brauer-Manin obstruction to weak approximation is the only one for quotients of $\SchSL_n$ by $G$. These new families include extensions of groups satisfying (BM) by kernels which are products of symmetric groups~$\fS_m$, with $m\neq 2,6$, and alternating groups $\fA_5$. We also investigate (BM) for small groups by giving an explicit list of small order groups for which (BM) is unknown and we show that for many of them (BM) holds under Schinzel’s hypothesis.

\end{abstract}

\maketitle
\section{Introduction}

The Grunwald problem is a stronger version of the Inverse Galois Problem (IGP) originating in the classical problem of determining which division algebras admit a $G$-crossed product structure for a given finite group $G$, see \cite{Schacher,Sonn, NP, HHK, RS}. 
Further recent interest in the problem arose from its relation to the regular IGP \cite{DG,KLN} and weak approximation \cite{Harari,MR3668767,HW1}. 


The problem asks for the possible local behaviour of Galois extensions~$L/k$ with finite Galois group $G$, that is, the possibilities for their completions $L_v/k_v$ at finitely many places~$v$ of $k$. It may be stated as follows. Given a finite set of places $S$ of a number field $k$, and Galois field extensions $L^{(v)}/k_v$ for~$v\in S$ with embeddings $\Gal(L^{(v)}/k_v)\ra G$, determine if there exists a Galois extension~$L/k$ with $\Gal(L/k)\simeq G$ such that $L_v\cong L^{(v)}$ for all $v\in S$? 

The Grunwald-Wang theorem \cite{MR33801} answers the problem affirmatively when $G$ is abelian and $S$ does not contain a place dividing $2$. Nevertheless, in \cite{MR0026992}, a counterexample with $k=\textbf{Q}$, $2\in S$ and $G=\mathbf{Z}/8\mathbf{Z}$ is given by Wang. In fact, Grunwald problems are expected to be solvable whenever the places of $S$ do not divide the order of $G$, a property also known as the {\it tame approximation property} for $G$ over $k$ \cite{MR3668767}. The answer at places dividing~$|G|$ is less clear, cf.\ \cite[\S3.1-3.2]{MR4458238} and \cite{RM} for recent work on certain groups~$G$. 

A main approach to the IGP and the Grunwald problem originates in Noether's construction, where $G$ is equipped with a free action on a rational space and the corresponding quotient is considered. 
Here we consider more generally a finite algebraic $k$-group $G$, choose an embedding $G\xhookrightarrow{\iota}\SchSL_{n,k}$, and let $X:=\SchSL_{n,k}/G$ be the quotient. 
The solvability of all Grunwald problems for a constant group $G$ is then equivalent to the \textit{weak approximation}~(WA) property on $X$ \cite{Harari}, that is, to the set $X(k)$ of $k$-rational points being dense in  $X(k_\Omega):=\prod_{v\in \Omega_k}X(k_v)$, where~$\Omega_k$ is the set of places of $k$ and the product is endowed with the product topology. 

In all known instances where (WA) fails, the failure is explained by the {\it Brauer--Manin set}. It is well known that $X(k)$ is contained in a closed subset $X(k_{\Omega})^{\Brun(X)}\subseteq X(k_\Omega)$ cut out by the Brauer--Manin obstruction, see \S\ref{sec:rational-BM}. It may happen that the latter containment is strict, so that (WA) fails. It is conjectured that $X(k)$ is in fact dense in $X(k_{\Omega})^{\Brun(X)}$ for $X$ as above. When this holds, 
we say that {\it the Brauer--Manin obstruction to weak approximation is the only one on~$X$} and 
that \textit{$G$ verifies (BM)}.
This does not depend on~$n$, nor on the embedding~$\iota$ by the "no-name" lemma \cite[Corollary~3.9]{MR2348904} and \cite[Proposition~13.3.11]{CTS}. 
Furthermore, if a constant group $G$ verifies (BM), 
then the tame approximation property holds for $G$ \cite[Corollary 6.3]{MR4014281}.\\    


The property (BM) is known for many solvable groups such as abelian $k$-groups by Borovoi \cite{MR1390687}, split extensions of $k$-groups that verify (BM) by an abelian group by Harari \cite{Harari}, 
and for finite supersolvable $k$-groups by Harpaz--Wittenberg \cite{HW1}. However, little is known about (BM) for non-solvable groups. In fact the known 
examples of such groups are the symmetric groups $\fS_n$, the alternating group $\fA_5$ \cite{MR1018955}, the simple group of Lie type $\mathrm{PSL}_2(\mathbf{F}_7)$ \cite{mestre2005correspondances} and, more generally, groups for which a generic polynomial exists, that is, when the quotient $X$ is retract rational \cite[Chapter~5]{JLY}. 

In the present paper, we show that group extensions of a group verifying (BM) by certain non-solvable kernels~$N$, such as powers of $\fA_5$ and $\fS_n$, also verify (BM), 
yielding new families of non-solvable groups verifying (BM) and having the tame approximation property.
\begin{thm}\label{bmgroups}
Let $k$ be a number field. Consider a short exact sequence of finite algebraic $k$-groups
\vskip -0.4cm
\begin{equation}\label{equ:exact}
\begin{tikzcd}[ampersand replacement=\&]
1 \arrow[r] \& N \arrow[r] \& E\arrow[r] \& Q\arrow[r] \& 1,
\end{tikzcd}
\end{equation}
with $N(\oline k) \simeq\fA_5^{t_0}\times\prod_{i=1}^s\fS_{n_i}^{t_i}$ 
for $s,t_0,t_1\dots,t_s\in \mathbb N\cup\{0\}$ and $n_1,\dots,n_s\in\mathbb N\setminus \{2,6\}$. 
If~$Q$ verifies (BM) over $k$, then $E$ also verifies (BM) over $k$. 

In particular, when $E$ is constant and $S$ is a finite set of primes coprime to $|E|$, all Grunwald problems for $E$ over $S$ are solvable, that is, $E$ has tame approximation over $k$.     
\end{thm}

In a parallel work, Harpaz and Wittenberg \cite[(1) of Corollary 4.13]{HW2}  prove independently -- with a different method -- the theorem  in case $Q$ is supersolvable and $N(\oline k) \simeq \fA_5$ or $\fS_n$ for $n\in \mathbb N\setminus\{2,6\}$. For supersolvable $Q$, their theorem \cite[Theorem 4.5]{HW2} actually allows $N$ to be any group that satisfies the following condition:
\begin{equation}
  \tag{$\bullet$}\label{conditionHW}
  \parbox{\dimexpr\linewidth-4em}{%
    \strut
    For every homogeneous space $X$ of $\SchSL_{n,k}$ and  $\overline{x}\in X(\overline{k})$ whose stabiliser~$H$ over~$\overline{x}$ satisfies $H(\overline{k})\simeq N(\overline{k})$, the Brauer-Manin obstruction to weak approximation is the only one on $X$.
    \strut
  }
\end{equation}
For the proof of Theorem \ref{bmgroups}, we do not require $Q$ to be supersolvable but on the other hand assume $Q$ verifies (BM) and reduce to the case where \eqref{equ:exact} splits. In such case, we show that $N$ can be chosen to be any group which satisfies (\ref{conditionHW}) with the additional condition on the $X$'s to have a rational point. For the purpose of proving this, we supply in Theorem~\ref{semidirectBM} a dévissage method for (BM) which extends 
Harari \cite[Théorème 1]{Harari} and whose proof relies on his fibration method \cite[Théorème~3]{Harari}. In particular, this approach allows showing that~(BM) holds for iterative semidirect products of groups verifying (BM) by abelian groups or by products of copies of~$\fA_5$,~$\fS_n$, $n\neq 2,6$.


To allow powers of $\fS_n$ and $\fA_5$ in the kernel,
we first consider more generally forms $G$ of powers of constant $k$-groups $N$ which are center-free and indecomposable, that is, which cannot be written as a product of two nontrivial groups. We show that such groups $G$ are Weil restrictions along a finite \'{e}tale algebra $A/k$ of forms of the extension of $N$ to $A$, see Theorem~\ref{twistprod}. This yields a description of homogeneous spaces of $\SchSL_r$ whose geometric stabilisers are powers of complete indecomposable center-free groups, see Theorem \ref{completethm}.

We then deduce that homogeneous spaces of $\SchSL_r$ with geometric stabiliser $\fS_n^t$,~
$n\neq 2,6$, are stably rational, see Corollary \ref{corsymm}.  
Similarly, we show that homogeneous spaces of~$\SchSL_r$ with geometric stabiliser $\fA_5^t$ are retract rational, see  Corollary \ref{corAlt}, extending  Maeda's theorem \cite{MR1018955} on the stable rationality of the field of invariants of~$\fA_5$.  To carry this out, we use a "twisted" version of an argument of Buhler  \cite[\S2.3, pp.46-47]{JLY} to produce a strongly versal $G$-torsor with a rational base for twisted forms $G$ of~$\fA_5^t$, $t\geq1$.

\subsubsection*{Small groups} We apply Theorem \ref{bmgroups} and previous known results in the literature to give an explicit list of the non-solvable groups with cardinality at most $500$ for which (BM) is known, see Proposition \ref{BMnonsoluble}. 
The list of those "small" non-solvable groups for which (BM) is unknown is given in Table \ref{noBMnonsoluble}. The smallest non-solvable groups for which~(BM) is unknown are central extensions of $\fA_5, \Sym_5$, and $\PSL_2(7)$ of orders $240$ and $336$. The smallest simple group for which (BM) is unknown is $\fA_6$.   
We also list 
the "small" (solvable) groups of cardinality at most $191$ for which (BM) is unknown, see Proposition~\ref{rsmallgroups}. 
Furthermore, under Schinzel's hypothesis (H)  \cite{MR0106202}, we show that all but one, a semidirect product of~
$\mathbf{Z}/4\mathbf{Z}$ acting on the Heizenberg group $\He_3$, verify (BM).
We suspect that a similar reasoning could show that $\He_3\rtimes \mathbf{Z}/4\mathbf{Z}$ verifies (BM)   conditionally on Schinzel's hypothesis. 

\subsubsection*{(BM) under Schinzel}
To get (BM) under the Schinzel hypothesis, we prove the following theorem. Recall that a finite $k$-group $Q$ is {\it supersolvable} if there exists a sequence
$\{1\}=Q_0\subseteq Q_1\subseteq\dots\subseteq Q_n = Q$
of normal $k$-subgroups of $Q$ such that~$Q_i/Q_{i-1}$ is cyclic for all $i$. 
\begin{thm}\label{bmschinzel}
Let $k$ be a number field. Consider a short exact sequence of finite algebraic $k$-groups
\begin{center}
\begin{tikzcd}[ampersand replacement=\&]
1 \arrow[r] \& N \arrow[r] \& E\arrow[r] \& Q\arrow[r] \& 1
\end{tikzcd}
\end{center}
such that the derived subgroup of $N(\oline k)$ is $\mathbf{Z}/2\mathbf{Z}$. Assume that either $Q$ is supersolvable, or the sequence splits and $Q$ verifies (BM). If Schinzel's hypothesis (H) holds, then $E$ verifies~(BM).
\end{thm}
Note that the class of supersolvable groups contains that of nilpotent groups. Further note that supersolvability differs from solvability by requiring that $Q_{i-1}$ is normal in $Q$ and not only  in $Q_i$. As an example, the alternating group $\mathfrak{A}_4$ is solvable, but not supersolvable.

To prove the theorem, we show  (BM) holds under the Schinzel's hypothesis (H) for homogeneous spaces of $\SchSL_r$ whose geometric stabiliser has derived subgroup $\mathbf{Z}/2\mathbf{Z}$, see Theorem \ref{metabelian}. For this, we combine the fibration method \cite[Chapitre 3, Corollaire~3.5]{MR2307807} that relies on Schinzel's hypothesis~(H), and a descent method for torsors under tori \cite[Corollaire 2.2]{HW1}. 
We then make use of Theorem \ref{semidirectBM} and the aforementioned theorem of Harpaz and Wittenberg \cite[Theorem~4.5]{HW2} to deduce Theorem \ref{bmschinzel}  in Section \ref{proofbmschinzel}.

\subsection*{Acknowledgement}
We thank Olivier Wittenberg for helpful discussions, noting the relation to Schinzel's hypothesis and shedding light on the proof we give in Appendix \ref{app:dem}. The first author was supported by a «Contrat doctoral sp\'ecifique normalien» from the \'Ecole normale sup\'erieure de Paris. The second author was supported by the Israel Science Foundation, grant no.~353/21. 

\section{Preliminaries}
\subsection{Notation}\label{notations}
Throughout the paper, if $k$ is a field, $\overline{k}$ denotes a fixed separable closure of $k$, and  $\Gamma_k$ the absolute Galois group $\Gal(\overline{k}/k)$. A variety over $k$ is a separated $k$-scheme of finite type.
If $G$ is an algebraic group over a perfect field $k$, a homogeneous space of $G$ is a $k$-variety~$X$ endowed with a left action of $G$ such that the action of $G(\overline{k})$ on $X(\overline{k})$ is transitive.

For a connected scheme $S$ and group $S$-schemes $G$ and $H$, we say that $H$ is an \textit{$S$-form} of~$G$ if there exists an \'{e}tale cover $T\rightarrow S$ such that the {\it base changes} $G_T$ and $H_T$ are $T$-isomorphic group schemes. When $S=\Spec(k)$, we say $H$ is a $k$-form of $G$. Recall that $S$-forms of $G$ are classified by the pointed set $\coh^1_{\et}(S,\underline{\Aut}(G))$ \cite[Chapter~III, \S4, p.134]{Milne} where~$\underline{\Aut}(G)$ is the sheaf of automorphisms of $G$.

When $S'\rightarrow S$ is a finite locally free morphism of schemes and $X'$ is a quasi-projective $S'$-scheme, we denote by $\Res_{S'/S}(X')$ the Weil restriction of $X'$ along $S'\rightarrow S$, which is again an $S$-scheme 
\cite[\S7.6, Theorem 4]{MR1045822}. 

\subsection{Torsors}\label{sec:tors}

Let $k$ be a field, $G$ an affine algebraic $k$-group and $f:Y\rightarrow X$ a morphism of $k$-varieties. Assume  a left (resp. right) action of $G$ on $Y$ is given, and let $G$ act trivially on $X$. We say that $f$ is a left (resp. right) torsor over $X$ when $f$ is a $G$-equivariant étale map and the morphism $G\times_XY\rightarrow Y\times_XY$ mapping $(g,y)$ to $(g.y,y)$ (resp. to $(y.g,y)$) is an isomorphism. When the context is clear, we omit the left or right nature of the torsor.

Furthermore, when $A$ is a $k$-algebra and $G$ a group $\Spec(A)$-scheme, we shall denote by~$\coh^1(A,G)$ the pointed set $\coh^1_{\et}(\Spec(A),G)$ defined in \cite[\S2.2.1]{CTS}. When $A=k$, this corresponds to the nonabelian Galois cohomology defined in \cite[Chapitre I, \S5]{Ser} and it classifies both right and left $G$-torsors over $k$. In this case, we let $\coc^1(k,G)$ denote the pointed set of $1$-cocycles, that is, continous maps $(a_{\gamma})_{\gamma\in\Gamma_k}:\Gamma_k\rightarrow E$ verifying $a_{\gamma\zeta}=a_{\gamma}\gamma(a_{\zeta})$ for any $\gamma,\zeta\in\Gamma_k$,  see \cite[Chapitre I, \S5.1]{MR0150130}. 


When $G$ is an affine algebraic group over a field $k$, $\sigma\in \coc^1(k,G)$ is a right $k$-torsor under~$G$, and~$Y$ a quasi-projective variety endowed with a left $G$-action, we denote by ${}_{\sigma}Y$ the quotient of~$\sigma\times_kY$ by the action of $G$ defined as~$g.(s,y)=(sg^{-1},gy)$ (the existence of such a variety is for instance ensured by \cite[Lemma~2.2.3]{Sko}). The quotient ${}_{\sigma}Y$ is also known as the \textit{contracted product} of $\sigma$ and $Y$. In particular, for a $k$-group $G$ and a right $k$-torsor $\sigma\in \coc^1(k,G)$ one gets the twisted group ${}_\sigma G$ when $G$ is acting on itself on the left by $g.h=ghg^{-1}$. Then, the twisted group ${}_\sigma G$ acts on ${}_\sigma Y$ in the following way: if $(s,g)$ is the class of $\alpha\in{}_\sigma G(\overline{k})$ and $(s',y)$ the class of $\beta\in{}_\sigma Y(\overline{k})$, then after choosing~$h\in G(\overline{k})$ such that $s'=s.h$, one may set $\alpha.\beta$ as the class of $(s,(gh).y)$ (also see \cite[p.20, Example~2]{Sko}).

Let $f:Y\rightarrow X$ be a left $G$-torsor, for a linear algebraic $k$-group $G$. It is said to be \textit{weakly versal} if, for every field extension $M/k$ and every left $G$-torsor $t:T\rightarrow\Spec(M)$, there exists an $M$-point $a$ of $X$ such that $t$ is the base change of $f$ by $a$. If, moreover, $a$ may be chosen in any nonempty Zariski open subset of $X$, then $f$ is said to be \textit{versal}. Following the terminology of~\cite[~\S1]{first2023highly} and the definitions in \cite[\S1]{MR3344763}, we say that $f$ is \textit{strongly versal} if there exists a finite dimensional $k$-vector space $V$ on which $G$ acts on the left faithfully and a $G$-equivariant dominant rational map $V\dashrightarrow Y$. According to \cite[Theorem $1.1$]{MR3344763}, we have the following implications:
\begin{center}
strongly versal $\Rightarrow$ versal $\Rightarrow$ weakly versal.
\end{center}

When $\underline{s}$ is a finite tuple of indeterminates, we say that a polynomial $f(\underline{s},x)\in k(\underline{s})[x]$ is a \textit{generic polynomial} over $k$ for a finite group $G$ if for every Galois extension $E/M$ with group $G$ and an overfield $M\supseteq k$, there exists $\underline{a}\in M^{|\underline{s}|}$ such that $f(\underline{a},x)$ is well defined and~$E$ is a splitting field of $f(\underline{a},x)$ over $M$. By \cite[Theorem 1]{MR1966633}, the existence of a generic polynomial for $G$ is equivalent to the existence of a weakly versal $G$-torsor whose base is a nonempty open subset of an affine space.

\subsection{Local-global principles and rationality}
\label{sec:rational-BM}
Over number fields $k$, we consider weak~approximation and the Brauer--Manin obstruction on a smooth $k$-variety $X$. 
Recall that~$X$~verifies \textit{weak approximation} off a finite subset $S\subseteq\Omega_k$ whenever $X(k)$ is dense in $\prod_{v\in\Omega_k\backslash S}X(k_v)$, the latter being endowed with the product topology. 
Let $\Br(X)=H^2_{\acute{\mathrm{e}}\mathrm{t}}(X,\mathbf{G}_m)$ denote the Brauer group of~$X$, and  $\Brun(X)$ the Brauer group of any smooth compactification of $X$ \cite[Proposition~3.7.10]{CTS}. Setting  $X(k_{\Omega})=\prod_{v\in\Omega_k}X(k_v)$, the \textit{Brauer-Manin pairing} has been introduced by Manin \cite{MR0427322} (see also \cite[\S1.3]{Harari}) and is given by:
\begin{center}
\begin{tikzcd}[row sep=0.5em]
\Brun(X)\times X(k_\Omega) \arrow[r] & \mathbf{Q}/\mathbf{Z} \\
\left((x_v),\alpha\right) \arrow[r,mapsto] & \sum_{v\in\Omega_K} inv_v\left(x_v^*(\alpha)\right),
\end{tikzcd}
\end{center}
where $inv_v:\Br(k_v)\rightarrow\mathbf{Q}/\mathbf{Z}$ is the local invariant and $x_v^*$ stands for the specialization morphism of Brauer groups $\Br(x_v):\Br(X)\rightarrow\Br(k_v)$. Elements of $X(k_{\Omega})$ orthogonal to~$\Brun(X)$ form a closed subset of $X(k_{\Omega})$ which contains $X(k)$ and which is called the Brauer-Manin set of~$X$. We shall say that {\it the Brauer--Manin obstruction to weak approximation is the only one on~$X$} if~$X(k)$ is dense in $X(k_{\Omega})^{\Brun(X)}$. In this situation, we will write that~$X$ verifies (BM).

Furthermore, over an arbitrary field $k$, we describe the properties of varieties considered throughout the paper. Two $k$-varieties~$X$ and $Y$ are said to be \textit{stably $k$-birational}, or \textit{stably $k$-birationally equivalent}, if there exist positive integers $m$ and $n$ such that $X\times\mathbf{P}^m_k$ and~$Y\times\mathbf{P}^n_k$ are $k$-birational. The variety~$X$ is said to be \textit{stably $k$-rational} if it is stably $k$-birational to~$\mathbf{P}^n_k$ for some positive integer $n$. We also say that the variety~$X$ is \textit{retract $k$-rational} if there exists a rational map $\mathbf{P}^n_k\dashrightarrow  X$ with a rational section, for some positive integer $n$. When there is no confusion, the underlying field will be omitted in the preceding terminologies. Weak approximation and (BM) are stably $k$-birational invariants.

As a ubiquitous statement in this article, we recall the statement of the "no-name" lemma~\cite[Corollary~3.9]{MR2348904}:
\begin{lem}["No-name" lemma]\label{noname}
Let $k$ be a field and $G$ a finite $k$-group. For  positive integers $r,s$  any embeddings $G\xhookrightarrow{}\SchSL_{r,k}$ and~$G\xhookrightarrow{}\SchSL_{s,k}$, the quotient varieties $\SchSL_{r,k}/G$ and~$\SchSL_{s,k}/G$ are stably $k$-birational.
\end{lem}


Also, for $k$-groups we consider throughout this paper the following properties:

\begin{nota}\label{varprop}
When $k$ is a field, consider the following properties on affine $k$-groups $G$:
\begin{enumerate}[label=(\alph*)]
\item there exists an embedding $G\xhookrightarrow{}\SchSL_{n,k}$ such that the variety $\SchSL_{n,k}/G$ is stably $k$-rational;
\item there exists a strongly-versal $G$-torsor $Y\rightarrow X$ over $k$, where $X$ is rational;
\item there exists an embedding $G\xhookrightarrow{}\SchSL_{n,k}$ such that the variety $\SchSL_{n,k}/G$ is retract $k$-rational;
\item the field $k$ is a number field and there exists an embedding $G\xhookrightarrow{}\SchSL_{n,k}$ such that the Brauer-Manin obstruction to weak approximation is the only one for the variety~$\SchSL_{n,k}/G$;
\item there exists a finite set of places $S$ of $k$ and an embedding $G\xhookrightarrow{}\SchSL_{n,k}$ such that the variety $\SchSL_{n,k}/G$ verifies weak approximation off $S$.
\end{enumerate}
\end{nota}
One may notice that properties (a), (c), (d) and (e) of Notation \ref{varprop} do not depend on the choice of the embedding, by Lemma \ref{noname} combined to the stably birational invariance of stable rationality, retract rationality, the Brauer-Manin obstruction to weak approximation and weak approximation off a finite set of places \cite[Proof of Proposition 13.2.3]{CTS}.
Furthermore, we have the following classical implications:
\begin{center}
(a) $\Rightarrow$ (b) $\Rightarrow$ (c) $\Rightarrow$ (d) $\Rightarrow$ (e)
\end{center}
where the first implication is given by choosing $Y\rightarrow X$ as being $f\times id_{\mathbf{P}^r_k}$ where $f:\SchSL_{n,k}\rightarrow\SchSL_{n,k}/G$ and $\mathbf{P}^r_k$ is chosen such that $\mathbf{P}^r_k\times\SchSL_{n,k}/G$ is rational. Moreover, the second implication is a consequence of the equivalence of $(1)$ and $(3)$ in \cite[Proposition~4.2]{MR3645070}.

\subsection{Preliminaries on group theory and almost complete stabilisers}\label{secacs}
Let $C_n$ denote the cyclic group of order $n$, and $\fS_n$ (resp. $\fA_n$) the symmetric (resp. alternating) group of degree $n$.

Let us recall that a finite group $G$ is \textit{complete} if it is center-free with no nontrivial outer automorphism. Examples of complete groups are the $\fS_n$'s for 
$n\neq2,6$. Furthermore, we will say that a group $G$ is \textit{indecomposable} if it cannot be written as a product of nontrivial groups: this is for example the case of $\fS_n$ and $\fA_n$ when $n\geq1$, since their normal subgroups have no direct factor.

Before going any further, let us state the following ubiquitous lemma on automorphisms of powers of indecomposable groups.
\begin{lem}[{\cite[Theorem 3.1]{Bidwell}}]\label{bidw}
    If $N$ is a group and $t$ a positive integer, then the following morphism is injective:
    \begin{center}
        \begin{tikzcd}[row sep=0.2, column sep=0.5cm]
            \iota:\Aut(N)^t\rtimes\fS_t\arrow[r] & \Aut(N^t)\\
            ((\varphi_i)_{i=1}^t,\sigma) \arrow[r,mapsto] & \left[(n_1,\dots,n_t)\mapsto \prod_{i=1}^t\varphi_i(n_{\sigma^{-1}(i)})\right],
        \end{tikzcd}
    \end{center}
    where $\fS_t$ acts on $\Aut(N)^t$ by $\sigma.(\varphi_i)_{i=1}^t=(\varphi_{\sigma^{-1}(i)})_{i=1}^t$ for any $\sigma\in\fS_t$ and $(\varphi_i)_{i=1}^t\in\Aut(N)^t$. Moreover, if $N$ is indecomposable and center-free, then $\iota$ is an isomorphism.
\end{lem}

Using the terminology of \cite[\S4.1.4]{HW2}, we shall say that a finite group $G$ is \textit{almost complete} if it is center-free and if the morphism $\Aut(G)\rightarrow\out(G)$ has a section. As these groups being ubiquitous in this article, we summarise  their basic properties in the following lemma. Complete groups are examples of almost complete groups. 

\begin{lem}\label{almostcomplete}
\begin{enumerate}[label=(\roman*)]
\item A finite group $G$ is almost complete if and only if any short exact sequence of profinite groups
\begin{center}
\begin{tikzcd}
1\arrow[r] & G \arrow[r] & E \arrow[r] & Q \arrow[r] & 1
\end{tikzcd}
\end{center}
splits as a semi-direct product of profinite groups $E\simeq G\rtimes Q$.
\item If $G$ is a finite almost complete indecomposable group, then for any positive integer~$t$, the group $G^t$ is almost complete.
\item If $H_1,\dots,H_r$ are finite almost complete groups where no pair of the $H_i$ have a common direct factor, then~$\prod_{i=1}^rH_i$ is almost complete.
\end{enumerate}
\end{lem}
\begin{proof}
Assertion (i) corresponds to \cite[Lemma 4.12.(1)]{HW2}. To prove (ii), first notice that~$G^t$ is center-free. Furthermore, since $G$ is indecomposable, Lemma \ref{bidw} ensures that $\Aut(G^t)\simeq \Aut(G)^t\rtimes\fS_t$ where $\fS_t$ acts on $\Aut(G)^t$ by permuting the coordinates. Denote by $s$ a section of $a:\Aut(G)\rightarrow\out(G)$. Then, the morphism $\Aut(G^t)\rightarrow\out(G^t)$ corresponds to the morphism $\prod_{1\leq i\leq t}a \times id_{\fS_t}:\Aut(G)^t\rtimes\fS_t\rightarrow\out(G)^t\rtimes\fS_t$ and a section is given by $\prod_{1\leq i\leq t}s\times id_{\fS_t}$, which prove (ii).

As for (iii), the group $\prod_{i=1}^rH_i$ is clearly center-free. Then \cite[Theorem 2.2]{Bidwell} ensures that $\Aut(\prod_{i=1}^rH_i)=\prod_{i=1}^r\Aut(H_i)$, so that the morphism $$\Aut\left(\prod_{i=1}^rH_i\right)\rightarrow\out\left(\prod_{i=1}^rH_i\right),$$ which coincides with  $\prod_{i=1}^r(\Aut(H_i)\rightarrow\out(H_i))$, has a section since each of the $H_i$'s is almost complete.
\end{proof}

Combining Lemma \ref{almostcomplete}.(i) with a theorem of P\'{a}l--Schlank one gets:

\begin{prop}\label{acrp}
Let $X$ be a homogeneous space of $\SchSL_{n,k}$ and $\overline{x}\in X(\overline{k})$. Denote by $G$ the stabiliser of $\overline{x}$. If the group $G(\overline{k})$ is finite and almost complete, then $X(k)\neq\emptyset$.
\end{prop}
\begin{proof}
By Lemma \ref{almostcomplete}.(i) the exact sequence
$$ 1\ra G(\oline k)\ra \pi^1_{\et}(X,\oline x)\ra \Gamma_k\ra 1,$$
splits. For homogeneous spaces of $\SchSL_n$, the existence of such a section implies the existence of a rational point by \cite[Theorem 7.6]{PS}. 
\end{proof}
Eventually, the following proposition gives a structural statement for homogeneous spaces whose stabilisers are products of almost complete groups with no pair having a common direct factor: 

\begin{prop}\label{prodeh}
Let $H_1,\dots,H_r$ be almost complete finite constant $k$-groups with no pair of the $H_i$'s having a common direct factor, and $m$ a positive integer. Suppose $X$ is an $\SchSL_{m,k}$-homogeneous space whose geometric stabiliser
is the direct product $\prod_{i=1}^rH_i$. Then~$X$ is stably birational to a product $\prod_{i=1}^rX_i$ where:
\begin{enumerate}[label=\alph*)]
\item there exists $n\in\mathbf{N}$ such that each $X_i$ is a $\SchSL_{n,k}$-homogeneous space;
\item for each $i\in\{1,\dots,r\}$, and for any $\overline{x_i}\in X_i(\overline{k})$, if one denotes by $G_i$ the stabiliser of $\overline{x_i}$, then $G_i(\overline{k})\simeq H_i$.
\end{enumerate}
\end{prop}
\begin{proof}
By Proposition \ref{acrp}, one has  $X(k)\neq\emptyset$. The choice of a rational point of~$X$ supplies a $\overline{k}/k$-form $G$ of the constant group $\prod_{i=1}^r H_i$ and an isomorphism of $\SchSL_{m,k}$-homogeneous spaces $X\simeq\SchSL_{m,k}/G$.

Furthermore, \cite[Theorem 2.2]{Bidwell} ensures that $\Aut(\prod_{i=1}^rH_i)=\prod_{i=1}^r\Aut(H_i)$. From this, the Galois action on $G(\overline{k})$ actually corresponds to a morphism $\Gamma_k\rightarrow\prod_{i=1}^r\Aut(H_i)$. For each $i\in\{1,\dots,r\}$, the projection $\Gamma_k\rightarrow\Aut(H_i)$ of the Galois action on $\Aut(H_i)$ defines a $\overline{k}/k$-form $\widetilde{H_i}$, so that $G\simeq\prod_{i=1}^r\widetilde{H_i}$. Now, choose a positive integer $n$ and, for each $i\in\{1,\dots,r\}$, an embedding $\widetilde{H_i}\xhookrightarrow{}\SchSL_{n,k}$. Then, Lemma \ref{noname} ensures that $\SchSL_{m,k}/G$ is stably birational to $\prod_{i=1}^r\SchSL_{n,k}/\widetilde{H_i}$, so that after setting $X_i= \SchSL_{n,k}/\widetilde{H_i}$, one gets the required statement.
\end{proof}

\subsection{Nonabelian Shapiro's map}\label{sec:shapiro}

If $\Gamma$ is a profinite group, a \textit{$\Gamma$-group} is a discrete group~$E$ endowed with a continuous action $\Gamma\xrightarrow{\varphi}\Aut(E)$. We shall also say that $\varphi$ yields a $\Gamma$-group, and  write~$g.e\coloneqq \varphi(g)(e)$ for $(g,e)\in\Gamma\times E$ when there is no confusion as to what $\varphi$ is. We then denote by $E\rtimes_{\varphi}\Gamma$ the semidirect product relative to $\varphi$ and by $\Sec_\Gamma(E\rtimes_{\varphi}\Gamma)$ the set of continuous sections of 
$E\rtimes_{\varphi}\Gamma\rightarrow \Gamma$. We use the following well-known correspondence:

\begin{fact}[{\cite[Lemma 7]{Stix}}]\label{seccoc}
    Let $\Gamma$ be a profinite group and $\varphi:\Gamma\rightarrow\Aut(E)$ a $\Gamma$-group. Denote by $\pr_1:E\rtimes_{\varphi}\Gamma\rightarrow E$ the first projection -- which is not necessarily a morphism. Then the following map is bijective
\begin{equation*}
    \begin{tikzcd}[row sep= 0.2]
        \Sec_\Gamma(E\rtimes_{\varphi}\Gamma)\arrow[r] & \coc^1(\Gamma,E) \\
        \sigma \arrow[r, mapsto] & \pr_1\circ\sigma
    \end{tikzcd}.
\end{equation*}
\end{fact}
For a closed subgroup $\Delta\leq \Gamma$ and a $\Delta$-group $E$, we let
$$\Ind_{\Delta}^{\Gamma}E\coloneqq\{f:\Gamma\rightarrow E\,|\,  f(dg)=df(g)\text{ for all }(d,g)\in \Delta\times \Gamma\}$$
denote the \textit{induced $\Gamma$-group} endowed with the left action
$(\gamma.f)(g)=f(g.\gamma)$ for $\gamma\in \Gamma$ and~$f\in \Ind_\Delta^\Gamma E$.
We may further identify $\Ind_{\Delta}^{\Gamma}E$ with $E^{\Delta\backslash\Gamma}$, where $\Gamma$ acts on $\Delta\backslash\Gamma$ on the right, via the following lemma.
\begin{lem}\label{identinduced}
    Let $\Gamma$ be a profinite group, $\Delta$ a closed subgroup of $\Gamma$ and $E$ a $\Delta$-group. Fix  a set of representatives $\{\eps_i:i\in\Delta\backslash\Gamma\}$ of $\Delta\backslash\Gamma$ and set
    \begin{center}
        \begin{tikzcd}[row sep=0.2]
            \omega:\Ind_{\Delta}^{\Gamma}E \arrow[r] & E^{\Delta\backslash\Gamma} \\
            f\arrow[r,mapsto] & (f(\eps_i))_{i\in\Delta\backslash\Gamma}
        \end{tikzcd}.
    \end{center}
    Furthermore for $i\in\Delta\backslash\Gamma$, $\gamma\in\Gamma$, write $\eps_i\gamma=\delta_i(\gamma)\eps_{i.\gamma}$ for a unique  $\delta_i(\gamma)\in \Delta$. Then:
    \begin{enumerate}
        \item  the map $\omega$ is an isomorphism of groups, which endows $E^{\Delta\backslash\Gamma}$ with a $\Gamma$-group structure defined as the composition of the $\Gamma$-group $\Gamma\rightarrow\Aut(\Ind_{\Delta}^{\Gamma}E)$ with
        \begin{equation*}
        \begin{tikzcd}[row sep=0.2]
            \Aut(\omega):\Aut(\Ind_{\Delta}^{\Gamma}E)\arrow[r] & \Aut(E^{\Delta\backslash\Gamma}) \\ 
            \varphi\arrow[r,mapsto] & \omega\varphi\omega^{-1}
        \end{tikzcd};
        \end{equation*}
        \item furthermore, the action of $\gamma\in\Gamma$ on $(e_i)_{i\in\Delta\backslash\Gamma}\in E^{\Delta\backslash\Gamma}$ described in (1) is given by
        $$\gamma.(e_i)_{i\in\Delta\backslash\Gamma}=
        (\delta_i(\gamma).e_{i.\gamma})_{i\in\Delta\backslash\Gamma};$$
        \item consider the $\Gamma$-group structure  on $E^{\Delta\backslash\Gamma}$ in (1) and let $1$ denote the trivial coset in~$\Delta\backslash\Gamma$. For $a\in\coc^1(\Gamma,E^{\Delta\backslash\Gamma})$, if we write $a_\gamma=(e_i(\gamma))_{i\in\Delta\backslash\Gamma}$, then:
$$e_i(\gamma)=\delta_1(\eps_i)^{-1}.\left[e_1(\eps_i)^{-1}e_1(\delta_i(\gamma))
\left(\delta_{1}(\delta_i(\gamma)).e_1(\eps_{i.\gamma})\right)
\right];$$
        \item if $\Delta$ acts trivially on $E$, then $\omega$ does not depend on the choice of the set of representatives of $\Delta\backslash\Gamma$. Furthermore, the conclusion of (3) may be rewritten as:
$$e_i(\gamma)=e_1(\eps_i)^{-1}e_1(\delta_i(\gamma))e_1(\eps_{i.\gamma}).$$
    \end{enumerate}
\end{lem}
\begin{proof}
    The proof of (1) and (2) is straightforward, and (4) follows immediately from (3). Let us prove (3). Since 
    $a$ is a cocycle, 
    $a_{\gamma\zeta}=a_{\gamma}\gamma(a_{\zeta})$ for $\gamma,\zeta\in\Gamma$, so that  (2) gives:
    \begin{equation}\label{cocycshapiro}
        e_i(\gamma\zeta)=e_i(\gamma)\bigl(\delta_i(\gamma).e_{i.\gamma}(\zeta)\bigr).
    \end{equation}
    Noting that
    $i.\eps_i^{-1}=1$,
    and applying (\ref{cocycshapiro}) with $(i,\gamma,\zeta)$ replaced by $(1,\eps_i,\gamma)$, we get:
    \begin{equation}\label{cocycshapiro1}
        e_1(\eps_i\gamma)=e_{i.\eps_i^{-1}}(\eps_i\gamma)=e_{1}(\eps_i)\bigl(\delta_{1}(\eps_i).e_i(\gamma)\bigr).
    \end{equation}
    Using  $i.\eps_i^{-1}=1$ again,  (\ref{cocycshapiro}) with $(i,\gamma,\zeta)$ replaced by
    $(1,\delta_i(\gamma),\eps_{i.\gamma})$ gives:
    \begin{equation}\label{cocycshapiro2}
        e_1\bigl(\delta_i(\gamma)\eps_{i.\gamma}\bigr)=e_{i.\eps_i^{-1}}\bigl(\delta_i(\gamma)\eps_{i.\gamma}\bigr)=
        e_1(\delta_i(\gamma))\bigl(\delta_1(\delta_i(\gamma)).e_1(\eps_{i.\gamma})\bigr).
    \end{equation}
    But since $\eps_i\gamma=\delta_i(\gamma)\eps_{i.\gamma}$, we can equate (\ref{cocycshapiro1}) and (\ref{cocycshapiro2}), yielding (3). 
\end{proof}
Returning to the general setting, note that the morphism of groups
$\Ind_{\Delta}^{\Gamma}E\rightarrow E$
mapping $f:\Gamma\rightarrow E$ to $f(1)$ is $\Delta$-equivariant, and hence induces \textit{Shapiro's map} of pointed sets
$$\sh:\coc^1(\Gamma,\Ind_{\Delta}^{\Gamma}E)\rightarrow\coc^1(\Delta,E)$$
which is known to be surjective by Shapiro's lemma \cite[Proof of Proposition 8]{Stix}.

Moreover, if $\Delta$ acts trivially on $E$, then  Shapiro's map coincides, via Fact \ref{seccoc}, with a map:
\begin{equation}\label{shapirotrivial}
    \sh':\Sec_\Gamma(\Ind^{\Gamma}_{\Delta}E\rtimes\Gamma)\rightarrow\Hom_{\cont}(\Delta,E).
\end{equation}
In this case, if we use notations of Lemma \ref{identinduced} and identify $\Ind_{\Delta}^{\Gamma}E$ with $E^{\Delta\backslash\Gamma}$ via $\omega$, the Shapiro map $\sh'$ can be described as follows:
\begin{equation}\label{shapirotrivialexplicit}
    \begin{tikzcd}[row sep=0.7, /tikz/column 1/.style={column sep=0.1},/tikz/column 2/.style={column sep=0.1}]
        \sh' & : & \Sec_\Gamma(E^{\Delta\backslash\Gamma}\rtimes\Gamma)\arrow[r] & \Hom_{\cont}(\Delta,E) \\
        & & \gamma\in\Gamma\mapsto((e_i(\gamma))_{i\in\Delta\backslash\Gamma},\gamma) \arrow[r,mapsto] & \delta\mapsto e_1(\delta)
    \end{tikzcd}.
\end{equation}

\section{Structure of forms for powers of groups}\label{secpg}

The following theorem is the main result regarding powers of groups. 
\begin{thm}\label{twistprod}
Let $N$ be a finite constant $k$-group and $t$ a positive integer. Assume further that $N$ is indecomposable, with trivial center. If an algebraic $k$-group~$G$ is a $k$-form of $N^t$, then there exists an \'{e}tale $k$-algebra $A$ of degree $t$ and a $k$-form $\widetilde{N_A}$ of $N_A\coloneqq N\otimes_kA$ such that $G=\Res_{A/k}(\widetilde{N_A})$.
\end{thm}


Before giving a proof, let us recall how Weil restrictions are related to induced $\Gamma_k$-groups.

\begin{lem}[{\cite[Theorem 1.3.1]{MR0670072}}]\label{weilinduced}
Let $L/k$ be a finite separable extension of fields. Fix an embedding $\sigma_0:L\xhookrightarrow{} \overline{k}$ where $\overline{k}$ is a separable closure of $k$ and set $\Gamma_L=\Gal(\overline{k}/L)$ which is a closed subgroup of~$\Gamma_k$. Then for any algebraic group $G$ over $L$, there exists a $\Gamma_k$-equivariant bijection
$$(\Res_{L/k}G)(\overline{k})\simeq\Ind_{\Gamma_L}^{\Gamma_k}(G(\overline{k}))$$
of $\Gamma_k$-groups, where $\Gamma_L$ acts on the left on $G(\overline{k})=\Hom_{\Spec(L)}(\Spec(\sigma_0),G)$  by composition, where $\Spec(\sigma_0)$ denotes the $\Spec(L)$-scheme $\Spec(\overline{k})$ corresponding to the embedding $\sigma_0$.
\end{lem}

We  also use the following technical lemma concerning Shapiro's map $\sh'$ from \eqref{shapirotrivial} when~$E$ is taken to be $\Aut(N)$. 
For this purpose, choosing a set of representatives of~$\Delta\backslash\Gamma$, we identify 
$\Aut(\Ind_\Delta^\Gamma N)$ with 
$\Aut(N^{\Delta\backslash\Gamma})$ via the isomorphism~
$\Aut(\omega)$ from Lemma \ref{identinduced}, and  $\Aut(N)^{\Delta\backslash\Gamma}\rtimes\Bij(\Delta\backslash\Gamma)$ with a subgroup of $\Aut(N^{\Delta\backslash\Gamma})=\Aut(\Ind_\Delta^\Gamma N)$ via Lemma \ref{bidw}. 
Eventually, we denote by
\begin{equation}\label{defrho}
\begin{tikzcd}[row sep=0.7]
    \rho:\Gamma\arrow[r] & \Bij(\Delta\backslash\Gamma) \\
    \gamma\arrow[r,mapsto] & \rho(\gamma):i\mapsto i.\gamma^{-1}
\end{tikzcd}
\end{equation}
the right action of $\Gamma$ on $\Delta\backslash\Gamma$. 
\begin{lem}\label{commutativeshapiro}
    Let $N$ be a finite group, $\Gamma$ a profinite group and $\Delta$ a closed subgroup of $\Gamma$ acting trivially on $\Aut(N)$. Let $\sigma\in\Sec_\Gamma(\Ind_{\Delta}^{\Gamma}\Aut(N)\rtimes\Gamma)$ and $\theta\coloneqq\sh'(\sigma)\in\Hom_{\cont}(\Delta,\Aut(N))$. Denote by $\Ind_{\Delta}^{\Gamma}N$ the $\Gamma$-group induced by $\theta$ and let $\beta:\Gamma\rightarrow\Aut(\Ind_{\Delta}^{\Gamma}N)$ the associated $\Gamma$-action. Then $\beta$ factors as: 
 \begin{center}
        \begin{tikzcd}[row sep=1.5cm, column sep= 0.8cm]
        \Gamma \arrow[r, "\sigma"] & 
        \Ind_{\Delta}^{\Gamma}\Aut(N)\rtimes\Gamma 
        \arrow[r, "(id{,}\rho)"] & 
        \Ind_{\Delta}^{\Gamma}\Aut(N)
        \rtimes\Bij(\Delta\backslash\Gamma) \arrow[r, "\psi"] &
        \Ind_{\Delta}^{\Gamma}\Aut(N)
        \rtimes\Bij(\Delta\backslash\Gamma)
    \end{tikzcd}
\end{center}
    for some inner automorphism $\psi$ of $\Aut(N)^{\Delta\backslash\Gamma}\rtimes \Bij(\Delta\backslash\Gamma)$.
\end{lem}
\begin{proof}
Let $\{\eps_i:i\in\Delta\backslash\Gamma\}$ be a set of representatives for $\Delta\backslash\Gamma$, and $\delta_i(\gamma)\in\Delta$ such that $\eps_i\gamma=\delta_i(\gamma)\eps_{i.\gamma}$. 
    Write $\sigma(\gamma)=((\varphi_i(\gamma))_{i\in\Delta\backslash\Gamma},\gamma)$ for $\gamma\in \Gamma$, 
    so that $\varphi:\Gamma\to\Aut(N)^{\Delta\backslash\Gamma}$, $\gamma\mapsto (\varphi_i(\gamma))_{i\in \Delta\backslash\Gamma}$ lies in 
$\coc^1(\Gamma,\Aut(N)^{\Delta\backslash\Gamma})$ by Fact \ref{seccoc}. Note that by Lemma~\ref{identinduced}.(3):
    \begin{equation}\label{exprcocycleproof}       \varphi_i(\gamma)=\varphi_{1}(\eps_i)^{-1}\circ\varphi_{1}(\delta_i(\gamma))\circ\varphi_1(\eps_{i.\gamma}),\text{ for }\gamma\in \Gamma,i\in \Delta\backslash\Gamma.
    \end{equation}
    Recall that $\beta$ induces a $\Gamma$-action on $N^{\Delta\backslash\Gamma}$ via $(\Aut(\omega)\circ\beta)(\gamma)=\omega\circ\beta(\gamma)\circ\omega^{-1}$. Moreover, by  Lemma \ref{identinduced}.(2), we have 
    \begin{equation}\label{down}
        \begin{tikzcd}[row sep=0.2]
            \omega\circ\beta(\gamma)\circ\omega^{-1}: N^{\Delta\backslash\Gamma}\arrow[r] & N^{\Delta\backslash\Gamma} \\
            (n_i)_{i\in\Delta\backslash\Gamma} \arrow[r,mapsto] & (\delta_i(\gamma).n_{i.\gamma})_{i\in\Delta\backslash\Gamma}
        \end{tikzcd}
    \end{equation}
    Since $\Delta$ acts on $N$ through $\theta$, and $\theta(\delta) =\varphi_1(\delta)$ by the explicit description (\ref{shapirotrivialexplicit}) of $\sh'$, morphism (\ref{down}) takes the form:
    \begin{equation*}\label{down1}
        \begin{tikzcd}[row sep=0.2]
            \omega\circ\beta(\gamma)\circ\omega^{-1}: N^{\Delta\backslash\Gamma}\arrow[r] & N^{\Delta\backslash\Gamma} \\
            (n_i)_{i\in\Delta\backslash\Gamma} \arrow[r,mapsto] & (\varphi_1(\delta_i(\gamma))(n_{i.\gamma}))_{i\in\Delta\backslash\Gamma}
        \end{tikzcd}.
    \end{equation*}
By definition of $\iota$ from Lemma \ref{bidw} and by definition of $\rho$ in (\ref{defrho}), this is none but the image by $\iota$ of the automorphism
$$\left(\bigl(\varphi_1(\delta_i(\gamma))\bigr)_{i\in\Delta\backslash\Gamma}, \rho(\gamma)\right)\in \Aut(N)^{\Delta\backslash\Gamma}\rtimes \Bij(\Delta\backslash\Gamma).$$
However, by \eqref{exprcocycleproof}, this image is conjugate by the element $(\varphi_1(\eps_i))_{i\in\Delta\backslash\Gamma}\in \Aut(N)^{\Delta\backslash\Gamma}$ to 
$$ ((id,\rho)\circ \sigma)(\gamma)=(\varphi_i(\gamma),\rho(\gamma))$$
which proves the statement.
\end{proof}

We may now combine Lemmas \ref{weilinduced} and \ref{commutativeshapiro} to give a proof of Theorem \ref{twistprod}.

\begin{proof}[Proof of Theorem \ref{twistprod}]
Let $G$ be a $k$-form of $N^t$. After choosing an isomorphism of groups $G(\overline{k})\simeq N^t$, the action of $\Gamma_k$ on $G(\overline{k})$ gives rise to a continuous action of $\Gamma_k$ on $N^t$, that is, to a morphism $\varphi:\Gamma_k\rightarrow\Aut(N^t)$. By Lemma \ref{bidw}, one has $\Aut(N^t)\simeq\Aut(N)^t\rtimes \fS_t$, where~$\mathfrak{S}_t$ acts on $\Aut(N)^t$ by permuting the coordinates.  The composition $\psi:\Gamma_k\rightarrow\fS_t$ of~$\varphi$ with the projection  $\Aut(N)^t\rtimes \fS_t\rightarrow\fS_t$ yields   a right action of $\Gamma_k$ on $\{1,\dots,t\}$ defined by $l.\gamma\coloneqq\psi(\gamma)^{-1}(l)$. Now pick  a set of representatives $I\subseteq\{1,\dots,t\}$ of the distinct orbits of $\{1,\dots,t\}$ under the right action of $\Gamma_k$, so that~$N^t=\prod_{i\in I}N^{i.\Gamma_k}$. Thus 
$\varphi$ factors as
\begin{equation}\label{prodforms}
    \Gamma_k\xrightarrow{\prod_{i\in I}\alpha_i}\prod_{i\in I}\Aut(N^{i.\Gamma_k})\xhookrightarrow{}\Aut(N^t)
\end{equation}
where $\alpha_i:\Gamma_k\rightarrow\Aut(N^{i.\Gamma_k})$ denotes the action of $\Gamma_k$ on the $i$-th factor $N^{i.\Gamma_k}$ of $N^t$. Let us denote by~$G_i$ the $k$-form of $N^{i.\Gamma_k}$ corresponding to $\alpha_i\in\coc^1(k,\Aut(N^{i.\Gamma_k}))$. By the factorisation~(\ref{prodforms}), we thus have $G=\prod_{i\in I}G_i$.

Next, we fix $i\in I$ and rewrite $G_i$ as a Weil restriction. For this purpose, we first use Lemma \ref{bidw} to identify $\Aut(N^{i.\Gamma_k})=\Aut(N)^{i.\Gamma_k}\rtimes\Bij(i.\Gamma_k)$. Furthermore, if we denote by~$\rho_i:\Gamma_k\ra\Bij(i.\Gamma_k)$ the right action described above,  
then $\alpha_i$ factors uniquely as
\begin{equation}\label{factoralpha}
\Gamma_k\xrightarrow{\sigma_i}\Aut(N)^{i.\Gamma_k}\rtimes\Gamma_k\hookrightarrow\Aut(N)^{i.\Gamma_k}\rtimes\Bij(i.\Gamma_k)
\end{equation}
for $\sigma_i\in\Sec_\Gamma(\Aut(N)^{i.\Gamma_k}\rtimes\Gamma_k)$. Letting $\Gamma_i$ be the stabiliser of $i$ under the right action of~$\Gamma_k$ on $i.\Gamma_k$, we may further identify:
\begin{equation*}\label{identification}
i.\Gamma_k\simeq\Gamma_i\backslash\Gamma_k.
\end{equation*}
The Shapiro map (\ref{shapirotrivial}) supplies a morphism $\theta_i\coloneqq\sh'(\sigma_i):\Gamma_i\ra \Aut(N)$. Denote by~$\Ind_{\Gamma_i}^{\Gamma_k}N$ the $\Gamma_i$-group induced by $\theta_i$ and $\beta_i:\Gamma_k\rightarrow\Aut(\Ind_{\Gamma_i}^{\Gamma_k}N)$ the corresponding action of $\Gamma_k$. 
Since $\alpha_i=(id,\rho)\circ \sigma$ by \eqref{factoralpha}, Lemma \ref{commutativeshapiro} yields an inner automorphism $\psi_i$ of $\Aut(N^{i.\Gamma_k})$ such that
\begin{equation}\label{equ00}
    \Aut(\omega_i)\circ\beta_i=\psi_i\circ\alpha_i.
\end{equation}
Now, choosing $\chi_i\in\Aut(N^{i.\Gamma_k})$ such that $\psi_i^{-1}$ is the conjugation map by $\chi_i$, equation (\ref{equ00}) yields the following commutative diagram:
\begin{equation*}
    \begin{tikzcd}
         & \Gamma_k \arrow[dl, "\beta_i"'] \arrow[dr,"\alpha_i"] & \\
         \Aut(\Ind^{\Gamma_k}_{\Gamma_i}N) \arrow[rr, "\Aut(\chi_i\circ\omega_i)"', "\sim"] & & \Aut(N^{i.\Gamma_k})
    \end{tikzcd}.
\end{equation*}
In other words, the map
\begin{equation}\label{gammaequivariance}
    \chi_i\circ\omega_i:\Ind_{\Gamma_i}^{\Gamma_k}N\rightarrow N^{i.\Gamma_k}
\end{equation}
is a $\Gamma_k$-equivariant isomorphism of groups, where $\Gamma_k$ acts on $\Ind_{\Gamma_i}^{\Gamma_k}N$ (resp. $N^{i.\Gamma_k}$) via~$\beta_i$ (resp. $\alpha_i$). Furthermore, recall that by definition of $G_i$, there is a $\Gamma_k$-equivariant isomorphism
\begin{equation}\label{ident1}
    G_i(\overline{k})\simeq N^{i.\Gamma_k}
\end{equation}
where the action of $\Gamma_k$ on the left-hand side (resp. on the right-hand side) is the Galois action on points (resp. is given by $\alpha_i$).
Also, if we set $M_i\coloneqq\overline{k}^{\Gamma_i}$ and $\widetilde{N_{M_i}}$ the $M_i$-form of $N$ associated to $\theta_i\in\coc^1(M_i,\Aut(N))$, then Lemma \ref{weilinduced} supplies a $\Gamma_k$-equivariant isomorphism of groups
\begin{equation}\label{ident2}
    \Res_{M_i/k}(\widetilde{N_{M_i}})(\overline{k})\simeq\Ind_{\Gamma_i}^{\Gamma_k}N.
\end{equation}
where the action of $\Gamma_k$ on the left-hand side (resp. right-hand side) is the Galois action on points (resp. is given by $\beta_i$). Now, if we combine (\ref{ident1}) and~(\ref{ident2}) with (\ref{gammaequivariance}), the map $\chi_i\circ\omega_i$ may be rewritten as a $\Gamma_k$-equivariant isomorphism
$$G_i(\overline{k})\simeq\Res_{M_i/k}(\widetilde{N_{M_i}})(\overline{k})$$
where the action of $\Gamma_k$ on both sides is the Galois action on points. By Galois descent, this extends to an isomorphism of algebraic groups $$G_i\simeq \Res_{M_i/k}(\widetilde{N_{M_i}}).$$
Thus, by setting $A\coloneqq\prod_{i\in I}M_i$ and $\widetilde{N_A}\coloneqq\prod_{i\in I}\widetilde{N_{M_i}}$, we eventually get
$$G\simeq\prod_{i\in I}\Res_{M_i/k}(\widetilde{N_{M_i}})=\Res_{A/k}(\widetilde{N_A})$$
which proves the statement.
\end{proof}
We deduce the following description of homogeneous spaces for powers of groups.
\begin{cor}\label{ehtwist}
Let  $N$ be a finite constant $k$-group, $t$ a positive integer and~$G$ a $k$-form of $N^t$. Assume further that $N$ is indecomposable with trivial center. Then there exist an \'{e}tale $k$-algebra $A$ of degree $t$ and a $k$-form $\widetilde{N_A}$ of $N_A$ such that, for every $k$-embedding $G\xhookrightarrow{}\SchSL_{r,k}$ and $A$-embedding $\widetilde{N_A}\xhookrightarrow{}\SchSL_{s,A}$, the varieties $\SchSL_{r,k}/G$ and $\Res_{A/k}(\SchSL_{s,A}/\widetilde{N_A})$ are stably birationally equivalent.
\end{cor}
\begin{proof}
By Theorem \ref{twistprod}, there exists an \'{e}tale $k$-algebra $A$ of degree $t$ and an $A$-form $\widetilde{N_A}$ of $N_A$ endowed with an isomorphism $G\simeq\Res_{A/k}(\widetilde{N_A})$. Let  $\widetilde{N_A}$ act faithfully on an affine space $\mathbf{A}^n_A$. By \cite[Corollary~A.5.4.(3)]{MR2723571}, one then has that
$$\Res_{A/k}\left(\mathbf{A}^n_A/\widetilde{N_A}\right)\simeq \left(\Res_{A/k}\left(\mathbf{A}^n_A\right)\right)/\Res_{A/k}\left(\widetilde{N_A}\right).$$
Since $\Res_{A/k}\left(\mathbf{A}^n_A\right)\simeq \mathbf{A}^{nt}_k$ and $\Res_{A/k}\left(\widetilde{N_A}\right)\simeq G$, one deduces from the latter that $\Res_{A/k}\left(\mathbf{A}^n_A/\widetilde{N_A}\right)\simeq \mathbf{A}^{nt}_k/G$. The statement then follows from Lemma \ref{noname}.
\end{proof}
Denote by $\mathcal{P}_k$ any of the properties on affine $k$-groups listed in Notation \ref{varprop}.
\begin{cor}\label{descent}
Let $t$ be a positive integer and $N$ a constant indecomposable finite $k$-group with trivial center. If every $L$-form of $N_L$ verifies~$\mathcal{P}_L$ for every separable field extension~$L/k$ with $[L:k]\leq t$,  then any $k$-form of $N^t$ verifies~$\mathcal{P}_k$.
\end{cor}
\begin{proof}
Let $t$ be a positive integer and $G$ a $k$-form of $N^t$. Let us now choose an \'{e}tale $k$-algebra~$A$ of degree $t$ and an $A$-form $\widetilde{N_A}$ of $N_A$ as in Corollary \ref{ehtwist}. After choosing a $k$-embedding $G\xhookrightarrow{}\SchSL_{r,k}$ and an $A$-embedding $\widetilde{N_A}\xhookrightarrow{}\SchSL_{s,A}$, Corollary \ref{ehtwist} ensures that~$\SchSL_{r,k}/G$ and $\Res_{A/k}\left(\SchSL_{s,A}/\widetilde{N_A}\right)$ are stably birational. Since $A$ is \'{e}tale over $k$, it may be written as $A=\prod_{1\leq i\leq n}M_i$ where  the field extensions $M_i/k$, $i=1,\ldots,r$ are separable of degree at most $t$. Moreover, $\coh^1(A,\Aut(N))=\prod_{1\leq i\leq n}\coh^1(M_i,\Aut(N))$ so that one may write $\widetilde{N_A}=\prod_{1\leq i\leq n}\widetilde{N_{M_i}}$, where the $M_i$-group $\widetilde{N_{M_i}}$ is an $M_i$-form of $N_{M_i}$ for $i=1,\ldots,n$. From this, we infer that $\SchSL_{r,k}/G$ is stably birational to the variety $$\prod_{1\leq i\leq n}\Res_{M_i/k}\left(\SchSL_{s,M_i}/\widetilde{N_{M_i}}\right).$$Furthermore, for each $i\in\{1,\dots,n\}$, the variety $\SchSL_{s,M_i}/\widetilde{N_{M_i}}$ is assumed to satisfy $\mathcal{P}_{M_i}$. Since the Weil restriction of a variety which is retract rational, resp.\ rational, resp.\ satisfies that the Brauer-Manin obstruction to weak approximation is the only one, resp.\ satisfies that weak approximation holds off a finite set of places, satisfies the same property (see \cite[Theorem 1.1]{MR4493005} for the Brauer-Manin obstruction to weak approximation, and \cite[Example 4.2]{MR2985010} for weak approximation), the statement of this proposition follows when $\mathcal{P}_k$ is one of the properties  (a), (c), (d) and (e) in Notation \ref{varprop}.

Assume now that $\mathcal{P}_k$ is property (b) in Notation \ref{varprop}. For each $i\in\{1,\dots,n\}$, choose a strongly-versal $\widetilde{N_{M_i}}$-torsor $Y_i\rightarrow X_i$ where $X_i$ is $M_i$-rational. By definition of strong versality, one may pick an $M_i$-vector space $V_i$, a faithful action of $\widetilde{N_{M_i}}$ on $V_i$ and an $\widetilde{N_{M_i}}$-equivariant rational dominant map $V_i\dashrightarrow Y_i$. After passing to Weil restrictions and products one gets a rational $G$-equivariant map $\prod_i\Res_{M_i/k}(V_i)\dashrightarrow \prod_i\Res_{M_i/k}(Y_i)$, which is dominant by \cite[Corollary~A.5.4.(1)]{MR2723571}, and by \cite[Corollary~A.5.4.(3)]{MR2723571} we get a $G$-torsor $$\pi:\prod_i\Res_{M_i/k}(Y_i)\rightarrow \prod_i\Res_{M_i/k}(X_i)$$ whose base is rational. Since $\prod_i\Res_{M_i/k}(V_i)$ is a $k$-vector space on which $G$ acts linearly and faithfully, $\pi$ is strongly versal, which proves the statement.
\end{proof}

\section{Complete groups} \label{secpcg}
The following theorem reduces our study of homogenous spaces whose geometric stabiliser is a power of a complete group to homogeneous spaces whose stabiliser is the group itself.
\begin{thm}\label{completethm}
Let  $H$ be an indecomposable complete constant finite $k$-group and~$t$ a positive integer. Consider a homogeneous space $X$ of $\SchSL_{n,k}$ and $\overline{x}\in X(\overline{k})$ whose stabiliser is denoted by $G$. If $G(\overline{k})\simeq H^t$, then there exists an \'{e}tale $k$-algebra $A$ of degree~$t$ such that for any $A$-embedding $H_A\xhookrightarrow{}\SchSL_{r,A}$, the varieties $X$ and $\Res_{A/k}(\SchSL_{r,A}/H_A)$ are stably $k$-birational.
\end{thm}

Note that here, as opposed to in Corollary \ref{ehtwist}, the group $H_A$ is embedded in $\SchSL_{r,A}$ as a constant group. Let us start by giving an application of Theorem \ref{completethm} when the given complete group is a symmetric group:
\begin{cor}\label{corsymm}
Let $k$ be a field, $X$ a homogeneous space of $\SchSL_{r,k}$ and $\overline{x}\in X(\overline{k})$. Denote by~$G$ the stabiliser of $\overline{x}$. If there exist positive integers $n,t$ with $n\not\in\{2,6\}$ and $G(\overline{k})\simeq \fS_n^t$, then $X$ is stably rational.
\end{cor}
\begin{proof}
First note that the symmetric group $\fS_n$ is complete \cite[Chapter IV, \S13, p.92]{MR0109842}. Further, since it is indecomposable, Theorem \ref{completethm} supplies an \'{e}tale $k$-algebra $A$ of degree~$t$ such that $X$ is stably birational to $\Res_{A/k}(\SchSL_{r,A}/(\fS_n)_A)$. The fundamental theorem of symmetric polynomials, combined with Lemma \ref{noname}, ensures that $\SchSL_{r,k}/(\fS_n)_k$ is stably $k$-rational, meaning that there exist nonnegative integers $r$ and $s$ and a birational $k$-map
\begin{center}\begin{tikzcd}\mathbf{P}^r_k\times_k\SchSL_{r,k}/(\fS_n)_k\arrow[r,dashed,"\sim"] & \mathbf{P}^s_k\end{tikzcd}.\end{center}
After tensoring by $\Spec(A)$ and passing to Weil restriction along $A/k$, we derive from the latter a birational $k$-map
\begin{center}\begin{tikzcd}\Res_{A/k}(\mathbf{P}^r_A)\times_k\Res_{A/k}(\SchSL_{r,A}/(\fS_n)_A)\arrow[r,dashed,"\sim"] & \Res_{A/k}(\mathbf{P}^s_A)\end{tikzcd}.\end{center}
Since $\Res_{A/k}(\mathbf{A}^r_A)$ (resp. $\Res_{A/k}(\mathbf{A}^s_A)$), which is an affine space, is an open subset of $\Res_{A/k}(\mathbf{P}^r_A)$ (resp. $\Res_{A/k}(\mathbf{P}^s_A)$), the latter is $k$-rational. From this, we conclude that $\Res_{A/k}(\SchSL_{r,A}/(\fS_n)_A)$ is stably $k$-rational, which proves the statement.
\end{proof}

For the proof of Theorem \ref{completethm}, we need the following lemma that classifies the stabilisers of a homogeneous space of $\SchSL_n$ in terms of the outer Galois action on one of them.
We use the following terminology. Given an isomorphism of finite groups $A\xrightarrow{\varphi} B$ and a continuous action of a profinite group $\Gamma$ on $A$, we shall call the action of $\sigma\in \Gamma$ on $B$, given by $b\mapsto \varphi(\sigma\varphi^{-1}(b))$,  {\it  the action induced by} $\varphi$.  Given two continuous actions $\alpha:\Gamma\ra\Aut(A)$ and $\beta:\Gamma\ra \Aut(B)$ we say that $\varphi$ is compatible with the outer actions of $\Gamma$ on $A$ and $B$ if the following diagram is commutative:
\begin{center}
\begin{tikzcd}
& \out(A)\arrow[dd, "\out(\varphi)"] \\
\Gamma \arrow[ur]\arrow[rd] & \\
& \out(B)
\end{tikzcd}
\end{center}
where the upper diagonal map (resp. the lower diagonal map) is the composition of $\alpha$ (resp.~$\beta$) 
with the quotient morphism $\Aut(A)\rightarrow\out(A)$ (resp. $\Aut(B)\rightarrow\out(B)$), and~$\out(\varphi)$ is induced by the isomorphism $\Aut(\varphi):\Aut(A)\rightarrow\Aut(B)$.

\begin{lem}\label{outnnlem}
Let $X$ be a homogeneous space of~$\SchSL_{N,k}$ and $x\in X(k)$ a rational point whose stabiliser is a finite group~$G$ over~$k$.
\begin{enumerate}[label=(\roman*)]
\item Let $b\in \SchSL_N(\overline{k})$ be such that $bx\in X(k)$. If $H$ denotes the stabiliser of $bx$, then $H=bGb^{-1}$. Moreover, via the isomorphism $H(\overline{k})\simeq G(\overline{k})$ mapping $h\in H(\overline{k})$ to~$b^{-1}hb$, the induced action of $\sigma\in \Gamma_k$ on $G(\overline{k})$ is given by:
$$g\mapsto b^{-1}\sigma(b)\sigma(g)\sigma(b^{-1})b\text{ for }g\in G(\oline k),\sigma\in \Gamma_k.$$
\item Conversely, let $H$ be a finite group over $k$ endowed with an isomorphism of finite groups $H(\overline{k})\simeq G(\overline{k})$ compatible with the outer action of~$\Gamma_k$. Assume further that~$G(\overline{k})$ has trivial center. Then there exists $y\in X(k)$ whose stabiliser is $H$.
\end{enumerate}
\end{lem}

To give a proof of this lemma, we remind the following fact from Galois cohomology.

\begin{remi}\label{sescocycle}
Consider a short exact sequence of discrete groups
\begin{center}
\begin{tikzcd}
1 \arrow[r] & H \arrow[r] & E \arrow[r] & Q \arrow[r] & 1
\end{tikzcd}

\end{center}
and $\Gamma$ a profinite group. Let $\alpha,\beta:\Gamma\rightarrow E$ be two continuous morphisms coequalised by the map $E\rightarrow Q$. Then the map $\sigma\in\Gamma\mapsto\beta(\sigma)\alpha(\sigma)^{-1}$ is a cocycle in $\coc^1(\Gamma,H)$ where the action of $\Gamma$ on $H$ is given by $\sigma.h=\alpha(\sigma)h\alpha(\sigma)^{-1}$, for any $\sigma\in\Gamma$ and $h\in H$.
\end{remi}
\begin{proof}
By assumption, $h_{\sigma}\coloneqq\beta(\sigma)\alpha(\sigma)^{-1}$ lies in $H$ for $\sigma\in\Gamma$. Then $(h_{\sigma})$ lies in $\coc^1(\Gamma,H)$ since, for all $\sigma,\tau\in\Gamma$, one has:
$$h_{\sigma\tau}=\beta(\sigma)\beta(\tau)\alpha(\tau)^{-1}\alpha(\sigma)^{-1}=\beta(\sigma)\alpha(\sigma)^{-1}\alpha(\sigma)\beta(\tau)\alpha(\tau^{-1})\alpha(\sigma)^{-1}=h_{\sigma}\alpha(\sigma)h_{\tau}\alpha(\sigma)^{-1}.$$
\end{proof}

\begin{proof}[{Proof of Lemma \ref{outnnlem}}]
For (i), note that the action induced by the isomorphism $\varphi:H(\oline k)\ra G(\oline k), h\mapsto bhb^{-1}$ is given by:
 $$\sigma\cdot g = \varphi\sigma\varphi^{-1}(g)=b^{-1}\sigma(bgb^{-1})b=b^{-1}\sigma(b)\sigma(g)\sigma(b^{-1})b, \text{ for }g\in G(\oline k),\sigma\in \Gamma_k.$$

Let  $H$ be as in the statement of (ii) and denote by $\varphi: H(\overline{k})\rightarrow G(\overline{k})$ the given isomorphism compatible with the outer action of $\Gamma_k$. Denote by $\alpha:\Gamma_k\rightarrow\Aut(G(\overline{k}))$ the Galois action on $G(\overline{k})$ and $\beta:\Gamma_k\rightarrow\Aut(G(\overline{k}))$ the action induced by applying $\varphi$ to the Galois action on~$H(\overline{k})$. Since $\varphi$ is compatible with the outer action of $\Gamma_k$, the map $\Aut(G(\overline{k}))\rightarrow\out(G(\overline{k}))$ is a coequaliser of $\alpha$ and $\beta$. As $G(\overline{k})$ has trivial center,  Reminder~\ref{sescocycle} applied to the short exact sequence
\begin{center}
\begin{tikzcd}
1 \arrow[r] & G(\overline{k}) \arrow[r,"\iota"] & \Aut(G(\overline{k})) \arrow[r] & \out(G(\overline{k})) \arrow[r] & 1
\end{tikzcd}
\end{center}
where $\iota$ maps $g$ to the associated inner automorphism, implies that $\sigma\mapsto h_{\sigma}\coloneqq \beta(\sigma)\alpha(\sigma)^{-1}$ lies in $\coc^1(k,G)$. Using the embedding $G(\overline{k})\xhookrightarrow{}\SchSL_N(\overline{k})$, this cocycles also lies in $\coc^1(k,\SchSL_N)$. Hilbert's theorem~$90$ hence supplies $b\in\SchSL_N(\overline{k})$ such that $h_{\sigma}=b^{-1}\sigma(b)$ for all $\sigma\in\Gamma_k$. From the fact that~$b^{-1}\sigma(b)$ is in the stabiliser $G(\overline{k})$ of $x\in X(k)$, we infer that $\sigma(bx)=bx$ for all $\sigma\in\Gamma_k$ so that $bx\in X(k)$.

To summarise, for $\sigma\in\Gamma_k$ and $g\in G(\overline{k})$, since $b^{-1}\sigma(b)$ acts on $G(\overline{k})$ via the inner automorphism $\iota(b^{-1}\sigma(b))$, we have:
$$\beta(\sigma)\circ\alpha(\sigma^{-1})(g)=b^{-1}\sigma(b)g\sigma(b^{-1})b$$
which boils down to
$$\beta(\sigma)(\sigma^{-1}(g))=b^{-1}\sigma(b)g\sigma(b^{-1})b.$$
This being true for all $g$, one eventually gets that for all $\sigma\in\Gamma_k$ and $g\in G(\overline{k})$, the following holds:
$$\beta(\sigma)(g)=b^{-1}\sigma(b)\sigma(g)\sigma(b^{-1})b.$$
Thus, by (i), $\beta$ coincides with the Galois action on the stabiliser of $bx$. By Galois descent, we  get an isomorphism of $H$ with the stabiliser of the rational point $bx$. 
\end{proof}
\begin{proof}[Proof of Theorem \ref{completethm}]
Since $H$ is complete, assertion (ii) of Lemma \ref{almostcomplete} ensures that $H^t$ is almost complete. From Proposition~\ref{acrp}, one then infers that $X$ has a rational point. Thus, there exists a $k$-form~$G$ of $H^t$, an embedding $G\xhookrightarrow{}\SchSL_{n,k}$ and an isomorphism of $\SchSL_{n,k}$-homogeneous spaces $X\simeq\SchSL_{n,k}/G$.

By Corollary \ref{ehtwist}, there exists an \'{e}tale $k$-algebra $A$ of degree $t$ and an $A$-form~$\widetilde{H_A}$ of $H_A$ such that for any embedding $\widetilde{H_A}\xhookrightarrow{}\SchSL_{r,A}$, the variety $X$ is stably birational to~$\Res_{A/k}(\SchSL_{r,A}/\widetilde{H_A})$. Now, since $A$ is \'{e}tale over $k$, it may be written as $A=\prod_{1\leq i\leq d}M_i$ where for $i\in\{1,\dots,d\}$, the field extension $M_i/k$ is separable and finite. Furthermore, since $\coh^1(A,\Aut(H))=\prod_{1\leq i\leq d}\coh^1(M_i,\Aut(H))$, one may write $\widetilde{H_A}=\prod_{1\leq i\leq d}\widetilde{H_{M_i}}$ where  the $M_i$-group $\widetilde{H_{M_i}}$ is an $M_i$-form of $H_{M_i}$ for $=1,\ldots,d$. From this, one deduces that
\begin{equation}\label{isomres}
\Res_{A/k}(\SchSL_{r,A}/\widetilde{H_A})= \prod_{i=1}^d\Res_{M_i/k}(\SchSL_{r,M_i}/\widetilde{H_{M_i}}).
\end{equation}
For $i\in\{1,\dots,d\}$, one may then choose an isomorphism 
\begin{equation}\label{isomapp}
\widetilde{H_{M_i}}(\overline{k})\simeq H(\overline{k})
\end{equation}
Since $H$ is complete, its outer automorphisms are trivial, so that isomorphism (\ref{isomapp}) is compatible with the outer Galois action of $\Gamma_k$. Lemma \ref{outnnlem}.(ii) may then be applied to ensure that~$H_{M_i}$ is embedded in $\SchSL_{r,M_i}$ and there exists an isomorphism of $\SchSL_{r,M_i}$-homogeneous spaces
\begin{equation}\label{isomhe}
\SchSL_{r,M_i}/\widetilde{H_{M_i}} \simeq \SchSL_{r,M_i}/H_{M_i}.
\end{equation}
The combination of (\ref{isomres}) and (\ref{isomhe}) then implies that $$\Res_{A/k}(\SchSL_{r,A}/\widetilde{H_A})\simeq \prod_{i=1}^d\Res_{M_i/k}(\SchSL_{r,M_i}/H_{M_i}).$$The latter may be rewritten as $\Res_{A/k}(\SchSL_{r,A}/\widetilde{H_A})\simeq \Res_{A/k}(\SchSL_{r,A}/H_A)$, to which $X$ is thus stably birational.
\end{proof}

\section{Forms of powers of $\fA_5$}\label{A5}

In this section, we assume that $k$ is a field of characteristic $0$ and we construct strongly-versal torsors for twisted forms of powers of $\fA_5$. 
Recall that $\fA_5$ is center-free and there is a commutative diagram
\begin{equation}\label{sesA5}
\begin{tikzcd}
\Aut(\fA_5)\arrow[r]\arrow[d,equal] & \out(\fA_5)\arrow[d,equal]\\
\fS_5\arrow[r] & C_2
\end{tikzcd}
\end{equation}
where the bottom map is the signature morphism. 

\begin{thm}\label{twistedAlternate}
Suppose $\charak k=0$,  let 
$t$ a positive integer, and $G$ a $k$-form of~$\fA_5^t$. Then there exists a strongly-versal $G$-torsor $Y\rightarrow X$ where $X$ is $k$-rational of dimension~$2t$.
\end{thm}

We start by stating the following application of Theorem \ref{twistedAlternate}:

\begin{cor}\label{corAlt}
Let $k$ be a field of characteristic zero, $X$ a homogeneous space of $\SchSL_{r,k}$ and $\overline{x}\in X(\overline{k})$. Denote by $N$ the stabiliser of $\overline{x}$. If there exists a positive integer $t$ such that $N(\overline{k})\simeq \fA_5^t$, then $X$ is retract rational.
\end{cor}
\begin{proof}
Since any transposition defines a section of (\ref{sesA5}), the group $\fA_5$ is almost complete by~(i) of Lemma \ref{almostcomplete}. Further, it is indecomposable so that $\fA_5^t$ is also almost complete by~(ii) of Lemma \ref{almostcomplete}. Proposition \ref{acrp} ensures that the variety~$X$ has a rational point, and the choice of such a point supplies a $k$-form $G$ of $\fA_5^t$ and an isomorphism of $\SchSL_{r,k}$-homogeneous spaces $X\simeq\SchSL_{r,k}/G$. Now, Theorem~\ref{twistedAlternate} gives a versal left $G$-torsor $Y\rightarrow X$ where $X$ is $k$-rational. From the equivalence of $(1)$ and $(3)$ in \cite[Proposition 4.2]{MR3645070}, one thus gets that $\SchSL_{r,k}/G$ is retract-rational.
\end{proof}



By Corollary \ref{descent}, applied with $\mathcal{P}_k$ as being property (e) of Notation \ref{varprop}, the proof of Theorem~\ref{twistedAlternate} amounts to the following proposition. Denote the base change of $\fA_5$ to a field~$M$ by $\fA_{5,M}$. 

\begin{prop}\label{easytwistedAlternate}
Let $M/k$ be a finite separable extension of fields and $(a_{\sigma})\in\coc^1(M,\fS_5)$. Let us denote by $\widetilde{\fA_5}$ the $M$-form of $\fA_{5,M}$ associated to the $1$-cocycle $(a_{\sigma})$ via the identification~$\fS_5=\Aut(\fA_5)$. Then there exists a strongly-versal left $\widetilde{\fA_5}$-torsor over an $M$-rational surface.
\end{prop}
The proof of Proposition \ref{easytwistedAlternate} follows that of \cite[Theorem 2.3.7]{JLY}, and requires the following lemmas. The proof of the first two lemmas may be skipped in a first reading.

The first lemma is a geometric reformulation of \cite[Theorem 1, equivalence of (1) and (4)]{MR1966633}, and is about extending splitting fields of generic polynomials to weakly versal torsors. 

\begin{lem}[{\cite[Theorem 1]{MR1966633}}]\label{torsorextension}
Let $G$ a finite constant group over $k$ and $f(\underline{s},x)\in k(\underline{s})[x]$ a generic polynomial for $G$. Denote by $L$ a splitting field of $f(\underline{s},x)$ over $k(\underline{s})$. Then, the $G$-extension~$L/k(\underline{s})$ extends to a weakly versal left $G$-torsor $Y\rightarrow X$, where $X$ is an open subset of $\mathbf{A}_k^{|\underline{s}|}$.
\end{lem}
The proof of Lemma \ref{torsorextension} given by DeMeyer and McKenzie in \cite[Proof of Theorem~1]{MR1966633} is ring-theoretic in flavour.  We supply a more geometric proof in Appendix~\ref{app:dem}.

The next lemma ensures that twisted linear actions remain linear:
\begin{lem}\label{twistsbgrp}
Let $G$ be a finite constant $k$-group, $H$ a normal subgroup of~$G$ and~$\sigma\in \coc^1(k,G)$. Consider a left $G$-variety $Y$ and a $G$-equivariant morphism $Y\rightarrow X$, where~$G$ acts trivially on $X$. Assume further that $Y\rightarrow X$ factors as $Y\rightarrow Z\rightarrow X$, where $Z$ is a left $G/H$-variety, $Y\rightarrow Z$ is $H$-equivariant, and $Z\rightarrow X$ is $G/H$-equivariant with $G/H$ acting trivially on $X$. Then:
\begin{enumerate}
\item The ${}_{\sigma}G$-equivariant morphism ${}_{\sigma}Y\rightarrow X$ factors canonically as ${}_{\sigma}Y\xrightarrow{} {}_{\tau}Z \xrightarrow{} X$, where $\tau$ is the $1$-cocycle given by the composition of $\sigma$ with the quotient morphism $G\rightarrow G/H$. Furthermore,  ${}_{\sigma}Y\rightarrow {}_{\tau}Z$ is $\widetilde{H}$-equivariant, where $\widetilde{H}$ is the $k$-form of~$H$ defined as the image of $\sigma$ by the map $\coc^1(k,G)\rightarrow\coc^1(k,\Aut(H))$ induced by the conjugation morphism $G\rightarrow\Aut(H)$. If moreover $Y\rightarrow X$, $Y\rightarrow Z$ and $Z\rightarrow X$ are respectively a $G$-torsor, an $H$-torsor, and a $G/H$-torsor, then any of the previous twisted maps is also a torsor.
\item Let $V$ be a $k$-vector space on which $G$ acts faithfully on the left, and assume further that the factorisation $Y\rightarrow Z\rightarrow X$ is $V\rightarrow V/H\rightarrow V/G$. In the notation of~(1), if $\sigma\in\coc^1(k,G)$, then~${}_{\sigma}V$ is a $k$-vector space on which ${}_{\sigma}G$ acts faithfully and linearly. In particular, the action of $\widetilde{H}$ on ${}_{\sigma}V$ is also faithful and linear, so that ${}_{\sigma}V\rightarrow_{\sigma}(V/H)$ is a versal $\widetilde{H}$-torsor over an irreducible open subscheme of ${}_{\sigma}(V/H)$.
\end{enumerate} 
\end{lem}
\begin{proof}
For the proof of (1), the canonical factorisation of ${}_{\sigma}Y\rightarrow X$ comes from the very definition of the contracted product. Since $\widetilde{H}$ is a subgroup of ${}_{\sigma}G$, it acts on $_{\sigma}Y$ and the map ${}_{\sigma}Y(\overline{k})\rightarrow X(\overline{k})$ is $\widetilde{H}(\overline{k})$-equivariant compatibly with the action of $\Gamma_k$.

Let us now prove (2). First note that the twisted form ${}_{\sigma}V$ of the $k$-vector space $V$ corresponds to a class in $\coh^1(k,\SchGL(V))$. Since the latter classifies $k$-vector spaces which are geometrically isomorphic to $V$, it ensures that ${}_{\sigma}V$ is naturally endowed with a $k$-vector space structure. Furthermore, after twisting the action morphism $a:G\times_kV\rightarrow V$, one gets an action $b:{_{\sigma}G}\times_{k}{_{\sigma}V}\rightarrow  {}_{\sigma}V$. Geometrically, $b$ corresponds to a morphism ${}_{\sigma}G(\overline{k})\rightarrow\SchGL({}_{\sigma}V)(\overline{k})$ which is compatible with the $\Gamma_k$-action. It thus  descends to a morphism ${}_{\sigma}G\rightarrow\SchGL({}_{\sigma}V)$ which defines $b$. This proves that ${}_{\sigma}G$ acts linearly on ${}_{\sigma}V$. The last part of the statement is inferred by $\widetilde{H}$ being a subgroup of ${}_{\sigma}G$. 
\end{proof}


The proof of Proposition \ref{easytwistedAlternate} uses the following setting. By
 \cite[\S2.2, Proposition~2.3.8]{JLY}, the polynomial
 $$f(s,u,x)=x^5+sx^3+u(x+1)\in k(s,u)[x],$$
 where $s,u$ and $x$ are indeterminates, is a generic polynomial of $\fS_5$ over $\mathbf{Q}$, hence over $k$. Denote by $L$ a splitting field of $f$ over $k(s,u)$. The genericity of $f$ ensures that $L/k(s,u)$ is an $\fS_5$-extension. We thus have the following tower of field extensions $$k(s,u)\subseteq k(s,u,\sqrt{\disc f}) \subseteq L$$
where $L$ is an $\fA_5$-extension of $k(s,u,\sqrt{\disc f})$. By Lemma \ref{torsorextension}, this tower of fields can be extended to morphisms of schemes $$Y\rightarrow Z\rightarrow X$$ where $Y\rightarrow X$ is a weakly versal left $\fS_5$-torsor, $Z=Y/\fA_5$ and $X$ is an open subset of $\mathbf{A}^2_k$.

The following lemma first ensures that $Y\rightarrow X$ is a strongly-versal $\fS_5$-torsor.

\begin{lem}\label{vvnontwisted}
For any vector space $V$ endowed with a faithful left linear action of $\fS_5$, there exists an $\fS_5$-equivariant dominant rational map $V\dashrightarrow Y$.
\end{lem}
\begin{proof}
Let $t$ be the map $V\rightarrow V/\fS_5$. There exist an integral open subset $U$ of $V/\fS_5$ such that, after setting $W=t^{-1}(U)$, the restriction of $t$ to $W$ is an $\fS_5$-torsor. Since $Y\rightarrow X$ is weakly versal, one may assume, after shrinking $U$, that there exist $\fS_5$-equivariant maps $r:W\rightarrow Y$ and $b:U\rightarrow X$ such that the following diagram is cartesian:
\begin{center}
\begin{tikzcd}
W \arrow[r, "r"]\arrow[d] & Y \arrow[d] \\
U \arrow[r, "b"] & X
\end{tikzcd}.
\end{center}
It remains to show that the map $b$ is dominant, which then implies that so is $r$. Indeed, otherwise  the closure $B$ of the image of $b$ is of dimension $0$ or $1$. 
If it were of dimension~$0$, then $b$ is constant. Since $U(k)\neq\emptyset$, this would mean that $B$ is a rational point of $X$, so that the $\fS_5$-torsor $Y_B\rightarrow B$ corresponds to a Galois extension of fields with group $\fS_5$. Furthermore, the versality of the $\fS_5$-torsor $W\rightarrow U$ would imply that $Y_B\rightarrow B$ is versal so that the trivial $\fS_5$-torsor over $k$ is a pullback of $Y_B\rightarrow B$. In particular, $Y_B$ consists in a $k$-point, hence $Y_B\rightarrow B$ is the identity morphism of $\Spec(k)$, which can not be.
Furthermore, if $B$ were of dimension $1$, then the curve~$C\coloneqq Y\times_XB$ would be dominated by $W$. It would thus be unirational, hence rational by the combination of \cite[Lemma 2.3]{MR1956057} and Lüroth's theorem \cite[\S4.6, Theorem~6.8]{MR1140705}. But $C\rightarrow B$ being an $\fS_5$-torsor, this supplies an embedding $$\fS_5\xhookrightarrow{} \Aut(C/B)=\Aut(k(C)/k(B))=\PGL_2(k(B)).$$ Nevertheless, as $k(B)$ is of characteristic $0$, there is no such embedding by \cite[Introduction]{MR2681719}\footnote{See also \cite[Corollary 1.3]{DKLN} for an explicit list of $1$-parameter generic polynomials over fields of characteristic $0$.}. From this, one eventually deduces that $B=X$, so that $b$ is dominant.
\end{proof}
We can now give a proof of Proposition \ref{easytwistedAlternate}:
\begin{proof}[Proof of Proposition \ref{easytwistedAlternate}]
Let $M/k$ be a finite extension, $(a_{\sigma})\in\coc^1(M,\fS_5)$ and $\widetilde{\fA_5}$ the $M$-form of $\fA_{5,M}$ associated to the $1$-cocycle $(a_{\sigma})$. Let $V$ be an $M$-vector space on which $\fS_5$ acts faithfully on the left. Now, Lemma \ref{vvnontwisted} supplies an $\fS_5$-equivariant dominant rational map~$V\dashrightarrow Y$. This yields a commutative diagram of rational maps:
\begin{center}
\begin{tikzcd}
V \arrow[r,dashed,"\alpha"]\arrow[d,"\fA_5",swap] & Y\arrow[d, "\fA_5",swap] \\
V/\fA_5 \arrow[r,dashed,"\beta"]\arrow[d,"C_2",swap] & Z \arrow[d,"C_2",swap]\\
V/\fS_5 \arrow[r,dashed] & X
\end{tikzcd}
\end{center}
where every horizontal map is dominant, $\alpha$ (resp. $\beta$) being $\fS_5$-equivariant (resp. $\fA_5$-equivariant) and each vertical map is a left torsor under the constant $M$-group written on its left. Using Lemma \ref{twistsbgrp}, one may twist such a diagram via the cocycle $(a_{\sigma})$, which ensures a commutative diagram of rational maps
\begin{equation}\label{vvtwisted}
\begin{tikzcd}
\widetilde{V} \arrow[r, dashed, "\widetilde{\alpha}"]\arrow[d,"\widetilde{\fA_5}",swap] & \widetilde{Y} \arrow[d,"\widetilde{\fA_5}",swap] \\
\widetilde{V}/\widetilde{\fA_5} \arrow[r, dashed, "\widetilde{\beta}"]\arrow[d,"C_2",swap] & \widetilde{Y}/\widetilde{\fA_5}\arrow[d,"C_2",swap]\\
V/\fS_5 \arrow[r, dashed] & X
\end{tikzcd}
\end{equation}
where each vertical map is a left torsor under the $M$-group written on its left. Moreover, any horizontal map of (\ref{vvtwisted}) is still dominant and $\widetilde{\alpha}$ (resp. $\widetilde{\beta}$) is $\widetilde{\fS_5}$-equivariant (resp. $C_2$-equivariant). Now, part (2) of Lemma \ref{twistsbgrp} ensures that $\widetilde{V}$ is an $M$-vector space on which~$\widetilde{\fA_5}$ acts faithfully and linearly. Thus, the dominance of $\widetilde{\alpha}$ yields the strong versality of the $\widetilde{\fA_5}$-torsor $\widetilde{Y}\rightarrow\widetilde{Y}/\widetilde{\fA_5}$.

To see that $\widetilde{Z}\coloneqq\widetilde{Y}/\widetilde{\fA_5}$ is $M$-rational, note that $\widetilde{Z}$ is none other than the twisted form of the $C_2$-torsor $Z\rightarrow X$ via the image $(\overline{a}_{\sigma}) \in \coc^1(M,C_2)$ of the $1$-cocycle $(a_{\sigma})\in \coc^1(M,\fS_5)$ induced by the quotient map $\fS_5\rightarrow\fS_5/\fA_5=C_2$. The $1$-cocycle $(\overline{a}_{\sigma})$ then corresponds to an extension $M(\sqrt{\alpha})/M$  for some $\alpha\in M$. By the very definition of twisted forms, the function field of $\widetilde{Z}$ is then the fixed subring of $M(s,u,\sqrt{\disc f})\otimes_{M(s,u)}M(s,u,\sqrt{\alpha})$ under the action of~$C_2$. If $\sqrt{\alpha}\in M$ the action of $C_2$ is trivial, so that
\begin{equation}\label{ff1}
    M(\widetilde{Z})=M(s,u,\sqrt{\disc f})\otimes_{M(s,u)}M(s,u,\sqrt{\alpha})=M(s,u,\sqrt{\disc f}).
\end{equation}
Let us now compute $M(\widetilde{Z})$ when $\sqrt{\alpha}\not\in M$. For this purpose, using the proof of \cite[Theorem 2.3.7]{JLY}, we have
$$\disc(f)=(108s^5 + 16s^4 u - 900s^3 u - 128s^2u^2 + 2000su^2 + 3125u^2 + 256u^3 )u^2$$
from which we deduce that $\disc(f)$ has odd degree. In particular, $\alpha.\disc(f)\not\in (M(s,u)^{\times})^2$ so that
\begin{equation}\label{ff2}
    M(\widetilde{Z})=\left(M(s,u,\sqrt{\disc f})\otimes_{M(s,u)}M(s,u,\sqrt{\alpha})\right)^{C_2}=M\left(s,u,\sqrt{\alpha.\disc(f)}\right)
\end{equation}
by the following simple lemma:
\begin{lem}\label{twistc2}
If $k(\sqrt{\alpha})/k$ (resp. $k(\sqrt{\beta})/k$) is a quadratic extension on which $C_2$ acts by mapping $\sqrt{\alpha}$ (resp.~$\sqrt{\beta}$) to $-\sqrt{\alpha}$ (resp.~$-\sqrt{\beta}$), then under the diagonal action of $C_2$ on~$k(\sqrt{\alpha})\otimes_kk(\sqrt{\beta})$, the $k$-algebra~$\left(k(\sqrt{\alpha})\otimes_kk(\sqrt{\beta})\right)^{C_2}$ of fixed elements may be described as follows:
$$\left(k(\sqrt{\alpha})\otimes_kk(\sqrt{\beta})\right)^{C_2}=
\begin{cases}  
k(\sqrt{\alpha\beta}) & \text{if } \alpha\beta\not\in (k^{\times})^2 \\
k\times k & \text{otherwise.}
\end{cases}   
$$
\end{lem}
The simple proof of the lemma is postponed to the end of the section. Thus, by comparing~(\ref{ff1}) and (\ref{ff2}), unconditionally on $\alpha\in M$ we have that:
$$M(\widetilde{Z})=M\left(s,u,\sqrt{\alpha.\disc(f)}\right).$$

It is thus enough to show that $M\left(s,u,\sqrt{\alpha.\disc(f)}\right)$ is $M$-rational. Indeed, following the proof of \cite[Theorem 2.3.7]{JLY}, one has:
$$\begin{array}{rl} 
\disc(f) & = (108s^5 + 16s^4u - 900s^3u - 128s^2u^2 + 2000su^2 + 256u^3 + 3125u^2)u^2 \\
& = \frac{1}{5^5}\left((5^5v+P)^2-4(9-20w)Q^2\right)s^{10}
\end{array}$$
where $P=1000w^2-450w+54$, $Q=(9-20w)(1-5w)$, $v=u^2/s^5$ and $w=u/s^2$. Hence, 
$$5\left(\frac{25\sqrt{\disc(f)}}{s^5Q}\right)^2=\left(\frac{5^5v+P}{Q}\right)^2-4(9-20w)$$
which, after multiplying by $\alpha$, leads to
$$5\left(\frac{25\sqrt{\alpha.\disc(f)}}{s^5Q}\right)^2=\alpha\left(\frac{5^5v+P}{Q}\right)^2-4\alpha(9-20w).$$
Thus, after setting $A \coloneqq \frac{25\sqrt{\alpha.\disc(f)}}{s^5Q}$ and $B\coloneqq \frac{5^5v+P}{Q}$, one gets
\begin{equation}\label{equABw}
5 A^2 = \alpha B^2 - 4\alpha(9-20w).
\end{equation}
The definition of $A$ and $B$ ensures that $M(s,u,\sqrt{\alpha.\disc(f)}) \supseteq M(A,B)$. Furthermore, the reverse inclusion holds since $w\in M(A,B)$ by (\ref{equABw}), and hence $v\in M(A,B)$ by the definition of $B$. 
But then $s = w^2/v$ and~$u = s^3v/w$ are in $M(A,B)$. 
In particular $\sqrt{\alpha.\disc(f)}\in M(A,B)$ by the very definition of $A$, and hence $M(s,u,\sqrt{\alpha.\disc(f)}) = M(A,B)$ is $M$-rational.
\end{proof}
\begin{proof}[Proof of Lemma \ref{twistc2}]
If $\alpha\beta\not\in (k^{\times})^2$, then $k(\sqrt{\alpha\beta})$ is a $k$-subalgebra of dimension $2$ of $k(\sqrt{\alpha})\otimes_kk(\sqrt{\beta})$ fixed by $C_2$. Since the $k$-algebra $\left(k(\sqrt{\alpha})\otimes_kk(\sqrt{\beta})\right)^{C_2}$ is of dimension $2$, this proves the statement in that case.

Now, if $\alpha\beta\in (k^{\times})^2$, then $k(\sqrt{\beta})=k(\sqrt{\alpha})$, from which one infers that $k(\sqrt{\alpha})\otimes_kk(\sqrt{\beta})=k(\sqrt{\alpha})\times k(\sqrt{\alpha})$ 
on which $C_2$ acts diagonally, so that $\left(k(\sqrt{\alpha})\otimes_kk(\sqrt{\beta})\right)^{C_2}=k\times k$.
\end{proof}

\section{Application to Grunwald problems for some non-solvable groups}

When $k$ is a number field, let us recall that the (BM) property for finite $k$-groups is a widely open question. In Section \ref{devissage}, we provide a criterion for an extension of finite $k$-groups to verify (BM). Afterwards, in Section \ref{nonsolvable}, we combine the results of previous sections with Theorem~\ref{semidirectBM} to infer a proof of Theorem \ref{bmgroups}, which supplies new families of non-solvable groups verifying~(BM). Eventually, Section \ref{secBMnonsoluble} is dedicated to  reviewing the non-solvable groups with cardinality at most $500$ for which (BM) holds.

\subsection{Extensions of groups verifying (BM)}\label{devissage}

The proof of the criterion is similar to the proof of \cite[Théorème 1]{Harari}, where Harari proves the case of a split exact sequence with abelian kernel. Our statement extends the latter:

\begin{thm}\label{semidirectBM}
Let $k$ be a number field and consider a short exact sequence of finite algebraic $k$-groups:
\begin{equation}\tag{$\mathcal{S}$}\label{sessemidirBM}
\begin{tikzcd}1\arrow[r] & N\arrow[r] & E\arrow[r] & Q\arrow[r] & 1.\end{tikzcd}
\end{equation}
Assume (BM) holds for  $Q$ and for every homogeneous space $X$ of $\SchSL_{r,k}$ admitting a point $\overline{x}\in X(\overline{k})$ whose stabiliser $S$ satisfies $S(\overline{k})\simeq N(\overline{k})$.  
Assume further that one of the following conditions is satisfied:
\begin{enumerate}[label=(\roman*)]
    \item for every field extension $K/k$, assume that $X(K)\neq\emptyset$ for every homogeneous space $X$ of $\SchSL_{r,K}$ admitting a point $\overline{x}\in X(\overline{K})$ whose stabiliser $S$ satisfies $S(\overline{K})\simeq N(\overline{K})$;
    \item the exact sequence (\ref{sessemidirBM}) is split.
\end{enumerate}
Then $E$ verifies~(BM).
\end{thm}
\begin{proof}
%
First choose an embedding of $E$ (resp. $Q$) in $\SchSL_m$ (resp. $\SchSL_n$) and if (\ref{sessemidirBM}) splits, fix a section $\sigma:Q\rightarrow E$ of (\ref{sessemidirBM}) and replace the embedding of $Q$ by the embedding:
$Q\xhookrightarrow{\sigma}E\xhookrightarrow{}\SchSL_n$ to assume $m=n$ in this case. Consider the projection morphism $p:\SchSL_m\times\SchSL_n\rightarrow\SchSL_n$. The algebraic group $E$ acts on the right on~$\SchSL_n$ via $Q$, which induces a diagonal action of $E$ on the right on $\SchSL_m\times\SchSL_n$. Then, $p$ is a (right) $E$-equivariant morphism, so that it induces a morphism
$$f:(\SchSL_m\times\SchSL_n)/E\rightarrow \SchSL_n/E=\SchSL_n/Q.$$
Let us apply the fibration method \cite[Théorème 3]{Harari} to $f$. For this purpose, we verify that $f$ satisfies the following conditions:
\begin{enumerate}[label=(\alph*)]
    \item the base variety $\SchSL_n/Q$ verifies (BM);
    \item any fibre of $f$ over a rational point verifies (BM);
    \item the generic fibre of $f$ is unirational.
\end{enumerate}
First note that (a) holds by assumption. 
For (b), fix a point $x$ of $\SchSL_n/Q$, denote by $\kappa$ its residue field, and let us describe $X\coloneqq f^{-1}(x)$. For this purpose, first note that $\SchSL_m$ acts on the left on $(\SchSL_m\times\SchSL_n)/E$ with~$h\in \SchSL_m$ acting on the class of $(a,b)\in\SchSL_m\times\SchSL_n$ by $h.(a,b)\coloneqq (ha,b)$. Moreover, when letting~$\SchSL_m$ act trivially on $\SchSL_n/Q$, the morphism $f$ is $\SchSL_m$-equivariant. In particular, this induces a left action of $\SchSL_{m,\kappa}$ on $X$. Now, let $\overline{\kappa}$ be an algebraic closure of $\kappa$ and choose~$s\in\SchSL_n(\overline{\kappa})$ a representative of the class $x\in \SchSL_n(\overline{\kappa})/Q(\overline{\kappa})$. Then $X(\overline{\kappa})$ is the subset of $(\SchSL_m(\overline{\kappa})\times\SchSL_n(\overline{\kappa}))/E(\overline{\kappa})$ consisting of classes of elements of the form $(g,s)$ with $g\in\SchSL_m(\overline{\kappa})$. Since $h\in\SchSL_m(\overline{\kappa})$ maps the class of $(g,s)$ to that of $(hg,s)$, the action of~$\SchSL_{m}(\overline{\kappa})$ on $X(\overline{\kappa})$ is transitive, so that $X$ is a homogeneous space of $\SchSL_{m,\kappa}$.

Besides, the stabiliser of the class of $(1,s)$ consists of elements $h\in\SchSL_m(\overline{\kappa})$ for which there exists $e\in E(\overline{\kappa})$ such that $(h,s)=(e,se)$. As $E$ acts on $\SchSL_n$ via $Q$, this equality holds if and only if $e\in N(\overline{\kappa})$, \textit{i.e.} if and only if $h\in N(\overline{\kappa})$. From this, we deduce that the stabiliser of the class of $(1,s)$ is $N(\overline{\kappa})$, and hence $X(\overline{\kappa})$, endowed with its $\SchSL_m(\overline{\kappa})$-action, is isomorphic to $\SchSL_m(\overline{\kappa})/N(\overline{\kappa})$.

Thus, if $x$ is a rational point of $\SchSL_n/Q$, then $f^{-1}(x)$ is a homogeneous space of $\SchSL_m$ whose geometric stabilisers are isomorphic to $N(\overline{k})$. Moreover, such homogeneous space verify (BM) by assumption, so that (b) is verified.\\

To verify (c) holds, first assume that (i) holds. If $x$ is the generic point of~$\SchSL_n/Q$, then we deduce from assumption (i) that the $\SchSL_m$-homogeneous space $f^{-1}(x)$ (whose geometric  stabiliser is isomorphic to $N(\oline k)$) has a rational point. If $y$ is such a rational point and $S$ denotes its stabiliser, it induces an isomorphism $\SchSL_m/S\xrightarrow{\sim} f^{-1}(x)$ defined by mapping the class of $g\in \SchSL_m$ to $g.y$. After composing with the quotient map $\SchSL_m\rightarrow\SchSL_m/S$, we obtain a dominant map $\SchSL_m\rightarrow f^{-1}(x)$, so that~$f^{-1}(x)$ is unirational.

Now, assume that (ii) holds. Since (\ref{sessemidirBM}) splits, the choice of embeddings we made in the beginning of the proof ensures that the diagonal embedding $\SchSL_n\rightarrow\SchSL_n\times\SchSL_n=\SchSL_m\times\SchSL_n$ is right $E$-equivariant. It thus induces a morphism $\SchSL_n/E\rightarrow(\SchSL_n\times\SchSL_n)/E$ which is a section of $f$. In particular, the generic fibre of $f$ has a rational point. The 
previous paragraph thus implies that the generic fibre of $f$ is unirational.

We may now apply \cite[Théorème 3]{Harari} to deduce that  $(\SchSL_m\times\SchSL_n)/E$ verifies (BM), but it remains to verify that it is stably birational to $\SchSL_m/E$. For this purpose, denote by $q:\SchSL_m\times\SchSL_n\rightarrow\SchSL_m$ the projection morphism, which is $E$-equivariant with respect to the right $E$-actions, so that it induces a quotient map $g:(\SchSL_m\times\SchSL_n)/E\rightarrow\SchSL_m/E$. First note that since $E$ acts freely on $\SchSL_m\times\SchSL_n$, then the following diagram, whose horizontal maps are the canonical quotient maps, is cartesian:
\begin{center}
\begin{tikzcd}[row sep=1.5cm]
\SchSL_m\times\SchSL_n\arrow[r]\arrow[d,"p"] & (\SchSL_m\times\SchSL_n)/E\arrow[d,"f"] \\
\SchSL_m\arrow[r] & \SchSL_m/E
\end{tikzcd}.
\end{center}
Since the map $\SchSL_m\rightarrow\SchSL_m/E$ is fppf, we deduce from this diagram that $g$ is a $\SchSL_m$-torsor for the fppf topology, hence for the Zariski topology by~\cite[Proposition 4.9]{Milne} (which is stated with $\SchGL_m$ although the proof is the same with $\SchSL_m$). In other words, the $\SchSL_m$-torsor~$g$ is locally trivial for the Zariski topology, yielding the claim.
\end{proof}

\subsection{Proof of Theorem \ref{bmgroups}}\label{nonsolvable}

Let us start by the following theorem, which is a combination of Proposition \ref{prodeh} with Corollaries \ref{corsymm} and~\ref{corAlt}. It gives the geometric nature of homogeneous spaces of $\SchSL_r$ whose geometric stabilisers are products of alternating and symmetric groups.

\begin{thm}\label{rateh}
Let $k$ be a field, $X$ a homogeneous space of~$\SchSL_{r,k}$ and $\overline{x}\in X(\overline{k})$. Denote by~$N$ the stabiliser of $\overline{x}$. 
\begin{enumerate}
\item If there exist distinct positive integers $n_1,\dots,n_s\not\in\{2,6\}$ and $t_1,\dots,t_s\in \mathbb N$ such that $$N(\overline{k})\simeq\prod_{i=1}^s\fS_{n_i}^{t_i},$$ then $X$ is stably rational.
\item If $k$ is of characteristic zero and if there exist distinct positive integers $n_1,\dots,n_s\not\in\{2,6\}$ and $t_0,\dots,t_s\in \mathbb N$ such that $$N(\overline{k})\simeq\fA_5^{t_0}\times \prod_{i=1}^s \fS_{n_i}^{t_i},$$ then $X$ is retract rational.
\end{enumerate}
\end{thm}
\begin{proof}
Let us first prove  $(1)$. By Proposition \ref{prodeh}, there exist a positive integer~$n$, and $\SchSL_{n,k}$-homogeneous spaces $X_1,\dots,X_s$ with points $\overline{x_i}\in X_i(\overline{k})$  such that $X$ is stably birational to~$\prod_{i=1}^sX_i$ and 
the stabilisers $G_i(\overline{k})$ of $\overline{x_i}$ are isomorphic to~$\fS_{n_i}^{t_i}$ for $i=1,\ldots,s$. Since the~$n_i$'s are different from $2$ and $6$, Corollary~\ref{corsymm} ensures that the $X_i$'s are stably rational. Thus $X$ is also stably rational.

To prove $(2)$, Proposition \ref{prodeh} supplies a positive integer $n$ and $\SchSL_{n,k}$-homogeneous spaces~$X_0,\dots,X_s$ such that $X$ is stably birational to $\prod_{i=1}^sX_i$, where the geometric stabilisers of $X_0$ (resp. $X_i$) is isomorphic to $\fA_5^{t_0}$ (resp. $\fS_{n_i}^{t_i}$ for all $i\in\{1,\dots,s\}$). Thus, Corollary~\ref{corAlt} ensures that $X_0$ is retract rational and Corollary \ref{corsymm} ensures that the $X_i$'s are stably rational for each $i\in\{1,\dots,s\}$, so that $X$ is also retract rational.
\end{proof}


When combining Theorems \ref{semidirectBM} and \ref{rateh}, we  get a proof of Theorem~\ref{bmgroups}:

\begin{proof}[Proof of Theorem \ref{bmgroups}]
The assertion follows from Theorem \ref{semidirectBM} once we verify its  assumption (i) holds. Indeed, by $(2)$ of Theorem \ref{rateh}, 
every homogeneous space of $\SchSL_{r,k}$, $r\geq 1$ with geometric stabiliser $N(\overline{k})$ is retract rational, so that the Brauer-Manin obstruction to weak approximation is the only one on it as in \S\ref{sec:rational-BM}, showing that (i) holds. 
\end{proof}




\subsection{The (BM) property for "small" non-solvable groups}\label{secBMnonsoluble}

In this section, we combine our results with those in the literature to  review  the non-solvable groups of order at most $500$ for which (BM) is known. We encapsulate this in the following proposition.

\begin{prop}\label{BMnonsoluble}
Over any number field $k$, the (BM) property holds for every finite non-solvable group whose cardinality is at most $500$, except perhaps for those appearing in Table~\ref{noBMnonsoluble}.
\end{prop}

To describe the groups in Tables \ref{noBMnonsoluble} we use the following notations. For finite groups~$G$,~$H$ and~$Q$, we write $G=H\rtimes Q$ for a semidirect product of $Q$ by $H$, and  $G=H.Q$ when~$G$ is a non-split extension of $Q$ by $H$. Let $D_n=C_n\rtimes C_2$ denote the dihedral group of order~$2n$. 
Let $Q_8$ denote the quaternion group, that is, the unique non-split extension of~$C_2$ by~$C_4$. Let $\GL[2](\mathbf{F}_q)$, $\SchSL_2(\mathbf{F}_q)$, $\PGL_2(\mathbf{F}_q)$, and $\PSL_2(\mathbf{F}_q)$ denote respectively the $2\times 2$ general, special,  projective general, and projective special linear groups  over $\mathbf{F}_q$.

\begin{center}
\begin{longtable}{|c|c|p{6cm}|}
\hline 
Order & Magma ID& Non-solvable groups for which (BM) is \textit{a priori} unknown \\
\hline 
\multirow{2}{*}{$240$} & 240,89 & $\mathrm{CSU}_2(\mathbf{F}_5)=C_2.\fS_5$\\
& 240,93 & $C_4.\fA_5$ \\
\hline \multirow{2}{*}{$336$} & 336,114 & $\SchSL_2(\mathbf{F}_7)$ \\
& 336,208 & $\mathrm{PGL}_2(\mathbf{F}_7)$ \\
\hline \multirow{1}{*}{$360$} & 360,118 & $\fA_6$\\
\hline \multirow{11}{*}{$480$} & 480,218 & $\SchGL_2(\mathbf{F}_5)$ \\
& 480,219 & $C_2^2.\fS_5$ \\
& 480,221 & $C_8.\fA_5$ \\
& 480,946 & $C_4.\fS_5$ \\
& 480,947 & $C_4.\fS_5$ \\
& 480,948 & $C_4.\fS_5$ \\
& 480,949 & $C_2\times \mathrm{CSU}_2(\mathbf{F}_5)$ \\
& 480,953 & $C_2^2.\fS_5$ \\
& 480,955 & $C_2\times (C_4.\fA_5)$ \\
& 480,957 & $D_4.\fA_5$ \\
& 480,959 & $Q_8.\fA_5$ \\
\hline \caption{List of non-solvable groups with cardinality at most~$500$ for which (BM) is unknown}
\label{noBMnonsoluble}
\end{longtable}
\end{center}

The proof of Proposition \ref{BMnonsoluble} makes use of Theorem \ref{bmgroups},  Harari's theorem \cite[Théorè\-me 1]{Harari} on split extensions of groups verifying~(BM) by abelian groups, and a theorem of Plans on the field of invariants of double covers of symmetric groups \cite[Theorem 11]{MR2517810}: 
\begin{manualtheorem}{A}[{\cite[Théorème 1]{Harari}}]\label{toolA} Given a split short exact sequence of finite groups
\begin{center}
\begin{tikzcd}
1 \arrow[r] & A \arrow[r] & E\arrow[r] & G\arrow[r] & 1
\end{tikzcd}
\end{center}
with $A$ abelian and $G$ verifying (BM), the group $E$ also verifies (BM).
\end{manualtheorem} 

\begin{manualtheorem}{B}[{\cite[Theorem 11]{MR2517810}}]\label{toolB}
Let $n\geq3$ be an odd integer and identify $\fS_{n-1}$ with the subgroup of $\fS_n$ fixing $n$. Consider a positive integer $r$ and an exact sequence of  groups:
\begin{center}
    \begin{tikzcd}
        1\arrow[r] & C_2^r \arrow[r] & H \arrow[r, "\pi"] & \fS_n \arrow[r] & 1
    \end{tikzcd}
\end{center}
such that the center of $H$ contains $C_2^r$. If we denote by $E$ the inverse image of $\fS_{n-1}$ by $\pi$, then $H$ verifies (BM) if and only if $E$ verifies (BM).
\end{manualtheorem}


\begin{manualtheorem}{C}\label{toolC}
Finite products of finite groups verifying (BM) also verify (BM).
\end{manualtheorem}

We also use the following results regarding the rationality of fields of invariants:
\begin{remi}\label{nonsolubleknown}
The following properties hold over $\mathbf{Q}$:
\begin{itemize}
\item All symmetric groups $\fS_n$ have stably rational fields of invariants.
\item The alternating group $\fA_5$ has stably rational fields of invariants by \cite{MR1018955}.
\item The group $\SchGL_2(\mathbf{F}_3)$ has stably rational fields of invariants by \cite{MR2357794}.
\item The  group $\SchSL_2(\mathbf{F}_5)$ has a generic polynomial by \cite{MR2318648}.
\item The simple group $\SchGL_3(\mathbf{F}_2)\simeq \mathrm{PSL}_2(7)$ has stably rational fields of invariants by \cite[Théorème~3]{mestre2005correspondances}.
\item The quaternion group $Q_8$ has stably rational fields of invariants by \cite{MR1550301}. 
\end{itemize}
In particular, all these groups verify (BM).
\end{remi}

\begin{proof}[Proof of Proposition \ref{BMnonsoluble}]
To prove the statement, we run the following algorithm, both by hand and using Magma,  over the non-solvable groups of cardinality $\leq 500$. 
To go over the list by hand, we use the online database of Dokchister \cite{Dokns},
where the groups are listed  via their Magma ID in the "Small Groups Library" of Magma. Our list is ordered by increasing Magma ID.

\textit{Step 1.} Following this order, determine all the groups $G$ for which (BM) follows from Tool \ref{toolA}, Tool \ref{toolC} and Reminder \ref{nonsolubleknown} by looking at all the short exact sequences in which~(I)~$G$ fits in the middle, the kernel is abelian, and (BM) is  known for the quotient by the reminder, or (II) $G$ is the product of two groups for which (BM) is known, one factor by the reminder and the other either by tool A or the reminder. The code in \cite["nonsolvable" code file]{Magma} outputs the minimal non-solvable groups for which (BM) remains unknown after ruling out the groups in (I), that is, those which are not extensions of a smaller group by an abelian kernel. For the smaller  groups, (BM) follows from the reminder. One additional group~$Q_8\times A_5$ is then ruled out in (II) as the product of two groups for which (BM) is known by Tool \ref{toolC}. 

\textit{Step 2.} Out of the remaining $22$ groups, we rule out two central extension $H$ of $S_5$, with Magma IDs 240,90-91, using Tool \ref{toolB}.  For these, \ref{toolB} yields an extension $E$ for which (BM) is known by the reminder and Proposition \ref{rsmallgroups}. We then rule out two more groups which are the direct products of these groups with $C_2$. 

\textit{Step 3.} Out of the remaining $18$ groups, we use Theorem \ref{bmgroups} to rule out two group extensions, $A_5\rtimes C_8$ and $A_5\rtimes Q_8$ with kernel $A_5$ and quotients satisfying (BM) by tools A 
The remaining $16$ groups appear in Table \ref{noBMnonsoluble}.
\end{proof}

\begin{rem}
Using Tool \ref{toolB}, we further reduce the verification of (BM) for the following groups to verifying (BM) for solvable groups of smaller order. We shall see in Section \ref{secmetabelian} that (BM) holds for these groups under  Schinzel's hypothesis H. 
\begin{itemize}
    \item The groups with Magma ID $240,89$ and $480,949$ satisfy (BM) if and only if the solvable group with Magma ID $48,28$ verifies (BM);
    \item The group with Magma ID $480,219$ satisfies (BM) if and only if the solvable group with Magma ID $96,66$ verifies (BM);
    \item The group with Magma ID $480,953$ satisfies (BM) if and only if the solvable group with Magma ID $96,190$ verifies (BM).
\end{itemize}
\end{rem}

\section{Metabelian stabilisers and Grunwald problems for small solvable groups}\label{secmetabelian}
Throughout this section $k$ is a number field. In the first part of this section, we prove that the Brauer-Manin obstruction to weak approximation is the only one for homogeneous spaces of $\SchSL_{n,k}$ whose geometric stabiliser has derived subgroup $C_2$, under Schinzel's hypothesis (H), and hence deriving Theorem \ref{bmschinzel}. 
In Section~\ref{smallgroups} we make use of Theorem~\ref{bmschinzel} and the results of Harari \cite{Harari} and Harpaz-Wittenberg \cite{HW1} to list the finite groups of order at most $191$ for which (BM)  is unknown.

\subsection{Homogeneous spaces with metabelian stabilisers}\label{metabstab}


The following is the main theorem on homogeneous spaces whose derived geometric stabilisers is $C_2$: 

\begin{thm}\label{metabelian}
Let $k$ be a number field, $X$ a homogeneous space of $\SchSL_{n,k}$ and $\overline{x}\in X(\overline{k})$. Assume that the stabiliser $G$ of $\overline{x}$ is a finite algebraic $\overline{k}$-group, and assume further that $G'= C_2$. If Schinzel's hypothesis (H) holds, then the Brauer-Manin obstruction to weak approximation is the only one on $X$.
\end{thm}

Let us recall that a reduced variety $X$ over a field $k$ is said to be \textit{split} if it contains an irreducible component which is geometrically irreducible \cite{MR1408492}. If $L/k$ is a field extension, one says that $X$ is split by $L$ if $X\otimes_kL$ is split. Equivalently, this means that there exists an irreducible component $Y$ of $X$ such that the algebraic closure of $k$ in $k(Y)$ embeds in~$L$.

Throughout the section  we denote by $\Pi^1_{\et}(X,\overline{x})$, for any geometric point $\overline{x}\in X(\overline{k})$, the following short exact sequence of profinite groups (see \cite[Théorème 6.1]{SGA1}):
\begin{center}
\begin{tikzcd}
1 \arrow[r] & \pi^1_{\et}(X\otimes_k\overline{k},\overline{x}) \arrow[r] & \pi^1_{\et}(X,\overline{x}) \arrow[r] & \Gamma_k \arrow[r] & 1.
\end{tikzcd}
\end{center}
In addition, we recall that if $S$ is a $k$-group of multiplicative type and $\widehat{S}$ its dual, then an $S$-torsor $f:Y\rightarrow X$ is called \textit{universal} if the homomorphism $\widehat{S}(\overline{k})\rightarrow\Pic(X_{\overline{k}})$, mapping a character $\chi:S_{\overline{k}}\rightarrow\mathbf{G}_{m,\overline{k}}$ to the image of $f$ by the pushfoward map $\chi^*:\coh^1(X_{\overline{k}},S_{\overline{k}})\rightarrow\coh^1(X_{\overline{k}},\mathbf{G}_m)=\Pic(X_{\overline{k}})$, is an isomorphism (see also \cite[p.25]{Sko}).

\begin{proof}[Proof of Theorem \ref{metabelian}]
Assume that $X(k_{\Omega})^{\Brun}\neq\emptyset$. Since $\SchSL_{n,k}$ is semisimple, Rosenlicht's lemma ensures that $\overline{k}[X_{\overline{k}}]^\times=\overline{k}^\times$. Furthermore, by \cite[\S5.2]{HW1}, one has $\Pic(X_{\overline{k}})=\Hom(G^{\ab}(\overline{k}),\overline{k}^\times)$ so that $\Pic(X_{\overline{k}})$ is finite. Thus, \cite[Proposition 6.1.4]{Sko} ensures that there exists a universal left $D$-torsor $f:Y\rightarrow X$ in the sense of \cite[\S2]{MR899402} and \cite[Definition~2.3.3]{Sko}, where $D(\overline{k})$ is the Cartier dual of $\Pic(X_{\overline{k}})$, hence $D(\overline{k})=G^{\ab}(\overline{k})$ so that~$D$ is a finite abelian group. In particular, $D$ is an algebraic group of multiplicative type and from \cite[Lemma 3.8]{HW2} one infers that $D$ fits in a short exact sequence
\begin{center}
\begin{tikzcd}
1 \arrow[r] & D\arrow[r] & T \arrow[r] & Q \arrow[r] & 1
\end{tikzcd}
\end{center}
where $T$ is a torus and $Q$ a quasi-trivial torus.

As in \S\ref{sec:tors}, let $_TY$
be the quotient of $T\times Y$ under the diagonal action of~$D$ defined for any point $d$ in $D$ by $(t,y)\mapsto (td^{-1},d.y)$, and denote by $_Tf:{}_TY\rightarrow X$ and $p:{}_TY\rightarrow Q$ the two projections. In order to prove that the Brauer-Manin obstruction to weak approximation is the only one for $X$, we are to apply the descent method for torsors under tori \cite[Corollaire 2.2]{HW1} to the left $T$-torsor $_Tf$. For this purpose, let us fix $\sigma\in\coc^1(k,T)$, and let us prove that the Brauer-Manin obstruction to weak approximation is the only one for the twisted variety~$_{\sigma}({}_TY)$.

After twisting the morphism $_TY\rightarrow Q$ by $\sigma$, one gets a morphism $g:{}_{\sigma}({}_TY)\rightarrow {}_{\sigma}Q$. Furthermore, since $Q$ is quasi-trivial, Shapiro's lemma combined with Hilbert's theorem~$90$ ensures that $_{\sigma}Q\simeq Q$, so that $Q$ is an open subset of an affine space $\mathbf{A}^d_k$. One may then choose a smooth compactification $Z$ of $_{\sigma}({}_TY)$ such that $g$ extends to a morphism $\overline{g}:Z\rightarrow\mathbf{P}^d_k$. Since $_{\sigma}({}_TY)$ is an open subset of $Z$, it is enough to prove that the Brauer-Manin obstruction to weak approximation is the only one for $Z$. We prove the latter by applying the fibration method \cite[Chapitre 3, Corollaire 3.5]{MR2307807} to the fibration $\overline{g}$. Indeed, the fibres of $g$ are twists of $f$ by an element of $\coc^1(k,D)$, that is, they are universal $D$-torsors over $X$. But \cite[Corollaire 5.4]{HW1} ensures that such universal torsors are homogeneous spaces of $\SchSL_{n,K}$ with geometric stabiliser~$C_2$. Thus, Lemma \ref{WAC2} below implies that the smooth fibres of $g$ over rational points of $_{\sigma}Q$ have the weak approximation property, and Lemma \ref{SplitC2} below guarantees that the fibres of $\overline{g}$ over points of codimension one of $\mathbf{P}^d_k$ are split by a quadratic extension. In other words, the assumptions of \cite[Chapitre 3, Corollaire~3.5]{MR2307807} are satisfied, when applied to the fibration $\overline{g}$, yielding the statement.
\end{proof}

\begin{lem}\label{WAC2}
Let $V$ be an $\SchSL_{n,k}$-homogeneous space with geometric stabiliser $C_2$. Then $V$ verifies weak approximation.
\end{lem}
\begin{proof}
Let us first note that $V$ verifies the Hasse principle. Indeed, assume that $V$ has local points and denote by $\overline{v}$ a geometric point of $V$. Note that $\Pi^1_{\et}(V,\overline{v})$ is isomorphic to
\begin{center}
\begin{tikzcd}
1\arrow[r] & C_2 \arrow[r] & \pi^1_{\et}(V,\overline{v}) \arrow[r] & \Gamma_k \arrow[r] & 1.
\end{tikzcd}
\end{center}
Since $V$ has local points, the class $[\Pi^1_{\et}(V,\overline{v})]\in\coh^2(k,C_2)={}_2\Br(k)$ is mapped to the trivial class in $\coh^2(k_v,C_2)={}_2\Br(k_v)$ for each $v\in\Omega_k$. The Brauer-Hasse-Noether exact sequence thus implies that $[\Pi^1_{\et}(V,\overline{v})]\in\coh^2(k,C_2)$ is trivial, so that $\Pi^1_{\et}(V,\overline{v})$ admits a section, which means that $V(k)\neq\emptyset$ by \cite[Theorem 7.6]{PS}.

Thus, if $V(k_{\Omega})\neq\emptyset$, one has  $V(k)\neq\emptyset$, so that $V\simeq \SchSL_{n,k}/C_2$. As $\mathbf{A}^2_k/C_2\cong \mathbf{A}^2_k$ is stably birational to $V$ by Lemma \ref{noname}, the variety $V$ is then $k$-stably rational, hence verifies the weak approximation property.
\end{proof}

\begin{lem}\label{SplitC2}
Let $R$ be a discrete valuation ring, $K$ its fraction field and $\kappa$ its residue field. Assume  $\kappa$ is of characteristic zero. Consider an $\SchSL_{n,K}$-homogeneous space $V$ with geometric stabiliser $C_2$ and let $\mathscr{V}$ be a regular proper $R$-scheme. If the generic fibre of~$\mathscr{V}$ contains $V$ as an open dense subscheme, then $\mathscr{V}\otimes_R\kappa$ is split by a quadratic extension of $\kappa$.
\end{lem}
\begin{proof}
First note that after replacing $\mathscr{V}$ by $\mathscr{V}\otimes_R\widehat{R}$, where $\widehat{R}$ denotes the completion of~$R$, one may assume that $R$ is complete, so that $R\simeq\kappa\llbracket t\rrbracket$ by \cite[Chapitre II, \S4, Théorème~2]{MR0150130}. Denote by $\partial_R:\coh^2(K,C_2)\rightarrow\coh^1(\kappa,C_2)$ the Serre residue map \cite[~\S 1.4.1]{CTS} and fix a geometric point $\overline{v}$ of $V$. Then the image of the class of $\Pi^1_{\et}(V,\overline{v})$ by $\partial_R$ corresponds to a quadratic extension $\kappa'/\kappa$, and we are to prove that~$\mathscr{V}\otimes_R\kappa$ is split by $\kappa'/\kappa$.

Since $R$ is henselian, \cite[Theorem II.3.10]{Milne} ensures that $\coh^1(R,C_2)=\coh^1(\kappa,C_2)$ so that~$\kappa'/\kappa$ is the special fibre of an \'{e}tale $R$-algebra $R'$. We replace $\mathscr{V}$ by $\mathscr{V}\otimes_RR'$ to assume  $\kappa=\kappa'$ and $R=R'$ and claim that $\mathscr{V}\otimes_R\kappa$ is split. Since $\kappa'=\kappa$, the class of $\Pi^1_{\et}(V,\overline{v})$ has a trivial residue. Now, \cite[Théorème~(b)]{Ducros} supplies a field extension~$\lambda$ of $\kappa$ such that~$\kappa$ is algebraically closed in $\lambda$ and $\cd(\lambda)\leq 1$. Set $S=\lambda\llbracket t\rrbracket$, so that the $R$-algebra $S$ is unramified. Furthermore, since $\kappa$ is algebraically closed in $\lambda$, one may also assume, after tensoring $\mathscr{V}$ by $S/R$, that $R=S$, i.e.\ that $\cd(\kappa)=1$.

Then, $[\Pi^1_{\et}(V,\overline{v})]\in\coh^2(K,C_2)={}_2\Br(K)$ having a trivial residue, \cite[Proposition~(2.1)]{MR244271} ensures that it comes from an element of $\Br(R)$. But $R$ being henselian, one infers from \cite[Theorem 3.4.2 (i)]{CTS} that $\Br(R)=\Br(\kappa)$ where $\Br(\kappa)=0$ since $\cd(\kappa)\leq 1$. Thus, the class of $\Pi^1_{\et}(V,\overline{v})$ in $\coh^2(K,C_2)$ is trivial, which means that $\Pi^1_{\et}(V,\overline{v})$ has a section. Hence \cite[Theorem 7.6]{PS} ensures that $V(K)\neq\emptyset$, from which one gets that $\mathscr{V}\otimes_R\kappa$ is split, using \cite[Lemma~1.1 (b)]{MR1408492}.
\end{proof}

\subsection{Proof of Theorem \ref{bmschinzel}}\label{proofbmschinzel}

We can now combine Theorem \ref{metabelian} with Theorem \ref{semidirectBM} and \cite[Theorem 4.5]{HW2} to supply a proof of Theorem~\ref{bmschinzel}.
\begin{proof}[Proof of Theorem \ref{bmschinzel}]
Assuming Schinzel's hypothesis (H),  Theorem \ref{metabelian} ensures that every homogeneous space of $\SchSL_r$ ($r\in \mathbb N$) whose geometric stabiliser has derived subgroup $C_2$ verifies~(BM).

Let us first consider the case where $Q$ is supersolvable and choose an embedding $E\xhookrightarrow{}\SchSL_r$. Then, we may apply \cite[Theorem 4.5]{HW2} to the homogeneous space $\SchSL_r/E$ and the embedding $N\subseteq E$. Indeed, if $\overline{x}$ is chosen to be the class of $1$ in $\SchSL_r/E$, then assumption~(1) of \cite[Theorem 4.5]{HW2} is fulfilled since the outer Galois action of $\Gamma_k$ on $E(\overline{k})$ factors through the Galois action $\Gamma_k$, which is trivial. Assumption (2) in \textit{ibidem} holds since $E/N=Q$ is supersolvable. Finally, by the first paragraph of this proof, (BM) holds for every homogeneous space $Y$ of $\SchSL_r$ with geometric stabiliser $N$ as required in ($\star$) of \textit{ibidem}. 

Now assume that the sequence of the statement is split and $Q$ verifies (BM). We may then apply Theorem \ref{semidirectBM}. For this purpose, first note that (ii) holds since the short exact sequence of the statement splits. Furthermore, the assumptions of \ref{semidirectBM} preceding~(i) and (ii) are automatically satisfied since $Q$ verifies (BM) and since every homogenous spaces with geometric stabiliser isomorphic to $N$ verifies (BM) by the first paragraph. 
\end{proof}

\subsection{The (BM) property for "small" solvable groups}\label{smallgroups}
Let us first state the main result of this subsection:

\begin{prop}\label{rsmallgroups}
Let $k$ be a number field. The (BM) property holds over $k$ for every finite group of cardinality at most $191$, except perhaps for those appearing in Tables~\ref{BMSH} and~\ref{noBM}. Furthermore, if one assumes that Schinzel's hypothesis (H) holds, then the~(BM) property holds for every finite group appearing in Table~\ref{BMSH}.
\end{prop}
To describe the groups in Tables \ref{BMSH} and \ref{noBM} we use the  notations from \S\ref{BMnonsoluble}. 
Further, let $\rm{Dic}_n$ denote the unique non-split extension $C_{2n}.C_2$ where $C_2$ acts on $C_{2n}$ by inversion.
For an odd prime $p$, let $\He_p$ denote the Heisenberg group, that is, the $p$-Sylow subgroup of~$\SchGL_3(\mathbf{F}_p)$.
Let $2^{1+2n}$ denote an extraspecial group whose center has order $2$,  the quotient being an elementary abelian group of rank $2n$.
\begin{center}
\begin{longtable}{|c|c|p{6cm}|p{5cm}|}
\hline Order & Magma ID & Rewriting that allows to prove (BM) & Method used to prove (BM), using the rewriting\\
\hline \multirow{2}{*}{$48$} & $48,28$ & $\mathrm{CSU}_2(\mathbf{F}_3)=Q_8.\fS_3$ & Theorem \ref{bmschinzel}\\
& $48,33$ & $Q_8.C_6$ & Theorem \ref{bmschinzel}\\
\hline $72$ & $72,3$ & $Q_8\rtimes C_9$ & Theorem \ref{bmschinzel}\\
\hline \multirow{12}{*}{$96$} & $96,3$ & $(C_2.C_4^2)\rtimes C_3$ & Theorem \ref{bmschinzel}\\
& $96,66$ & $Q_8\rtimes\mathrm{Dic}_3$ & Theorem \ref{bmschinzel}\\
& $96,67$ & $Q_8.\mathrm{Dic}_3$ & Theorem \ref{bmschinzel}\\
& $96,74$ & $Q_8.C_{12}$ & Theorem \ref{bmschinzel}\\
& $96,188$ & $C_2\times\mathrm{CSU}_2(\mathbf{F}_3)$ & Theorem \ref{bmschinzel} and Tool \ref{toolC}\\
& $96,190$ & $Q_8.D_6$ & Theorem \ref{bmschinzel}\\
& $96,191$ & $Q_8.D_6$ & Theorem \ref{bmschinzel}\\
& $96,192$ & $Q_8.D_6$ & Theorem \ref{bmschinzel}\\
& $96,193$ & $Q_8.D_6$ & Theorem \ref{bmschinzel}\\
& $96,200$ & $C_2\times(Q_8.C_6)$ & Theorem \ref{bmschinzel} and Tool \ref{toolC}\\
& $96,201$ & $Q_8.(C_2\times C_6)$ & Theorem \ref{bmschinzel}\\
& $96,202$ & $Q_8.(C_2\times C_6)$ & Theorem \ref{bmschinzel}\\
\hline \multirow{6}{*}{144} & $144,31$ & $Q_8.D_9$ & Theorem \ref{bmschinzel}\\
& $144,32$ & $Q_8\rtimes D_9$ & Theorem \ref{bmschinzel}\\
& $144,35$ & $C_2\times(Q_8\rtimes C_9)$ & Theorem \ref{bmschinzel} and Tool \ref{toolC}\\
& $144,36$ & $Q_8.C_{18}$ & Theorem \ref{bmschinzel}\\
& $144,121$ & $C_3\times\mathrm{CSU}_2(\mathbf{F}_3)$ & Theorem \ref{bmschinzel} and Tool \ref{toolC}\\
& $144,124$ & $Q_8.(C_3\rtimes\fS_3)$ & Theorem \ref{bmschinzel}\\
& $144,127$ & $Q_8.(C_3\times\fS_3)$ & Theorem \ref{bmschinzel}\\
& $144,157$ & $C_3\times(C_4.\fA_4)$ & Theorem \ref{bmschinzel} and Tool \ref{toolC}\\
\hline $160$ & $160,199$ & $2^{1+4}\rtimes C_5$ & Theorem \ref{bmschinzel} \\
\hline
\caption{List of finite groups with order less than $191$, for which (BM) is known conditionally under Schinzel's hypothesis (H), but not unconditionally}\label{BMSH}
\end{longtable}
\end{center}

\begin{center}
\begin{longtable}{|c|c|c|}
\hline Order & Magma ID & Groups for which (BM) is \textit{a priori} unknown \\
\hline $108$ & $108,15$ & $\He_3\rtimes C_4$ \\
\hline \caption{The only group of order less than $191$ for which~(BM) is \textit{a priori} unknown}\label{noBM}
\end{longtable}
\end{center}

To prove Proposition \ref{rsmallgroups}, we append to Tools \ref{toolA}, \ref{toolB} and \ref{toolC} 
a result of Harpaz-Wittenberg: 

\begin{manualtheorem}{D}[{\cite[Th\'eor\`eme B]{HW1}}]\label{toolD}
Any  finite supersolvable group verifies~(BM).
\end{manualtheorem}

%



We also use the fact that the solvable group $\SchSL_2(\mathbf{F}_3)$ verifies (BM). The stable rationality of its fields of invariants was proved by Rikuna in an unpublished paper. The retract rationality of such fields was  later published by Burdick and Jonker \cite[Theorem~3.6]{MR3071247}: 

\begin{remi}[{\cite[Theorem 3.6]{MR3071247}}]\label{remisl23}
    The fields of invariants of $\SchSL_2(\mathbf{F}_3)$ are retract rational. In particular, $\SchSL_2(\mathbf{F}_3)$ satisfies (BM).
\end{remi}

\begin{proof}[Proof of Proposition \ref{rsmallgroups}]
The proof is derived from  Tools \ref{toolA}, \ref{toolC}, \ref{toolD}, Reminders \ref{nonsolubleknown}, \ref{remisl23} and Theorem \ref{bmschinzel} as follows:\\
\textit{Step 1.} In our Magma code \cite["solvable" code file]{Magma}, we use the "Small Groups Library" of Magma to produce the list of finite groups of cardinality at most $191$ for which neither Tool \ref{toolA} nor Tool \ref{toolC} nor Tool \ref{toolD} might be used to prove (BM). 

\noindent \textit{Step 2.} For any of the Magma IDs of this list, we look at all the short exact sequences in which the corresponding group fits in the middle, via an online database of Dokchister~\cite{Dok23}. We then eliminate, by hand, all the groups among that list for which (BM) is known unconditionally, via a combination of Reminders \ref{nonsolubleknown} and \ref{remisl23} with  Tools \ref{toolA} and \ref{toolC}. The remaining groups are listed in Tables \ref{BMSH} and \ref{noBM}. 



\noindent \textit{Step 3.} Eventually, we determine, among the remaining groups, those for which (BM) modulo Schinzel's hypothesis (H) can be proved using Theorem \ref{bmschinzel}. The remaining group is listed in Table \ref{noBM}. 
\end{proof}

\appendix

\section{Geometric proof of Lemma \ref{torsorextension}}\label{app:dem}
We include here a geometric proof of the equivalence of (1) and (4) in \cite[Theorem~1]{MR1966633}, that is, a proof of Lemma \ref{torsorextension}, of which we use the notations.

We start by setting $K\coloneqq k(\underline{s})[x]/f(\underline{s},x)$, $A\coloneqq k\left[\underline{s},\frac{1}{\den(f).\disc(f)}\right]$, $X\coloneqq\Spec(A)$ and $Z\coloneqq\Spec(A[x]/f(\underline{s},x))$. First note that since $f$ is monic in $x$,  the natural morphism $\varphi:Z\rightarrow X$ is finite and flat. Besides, since $X$ is an open subscheme of $\mathbf{A}_k^{|\underline{s}|}$ where~$\disc(f)$ does not vanish, the fibres of $f$ are \'{e}tale. Thus $f$ is flat with \'{e}tale fibres, hence \'{e}tale. Letting $\overline{k(\underline{s})}$ be a separable closure of $k(\underline{s})$, we denote by $\eta:\Spec(k(\underline{s}))\rightarrow X$ the generic point of $X$ and $\overline{\eta}$ its composition with $\zeta:\Spec(\overline{k(\underline{s})})\rightarrow\Spec(k(\underline{s}))$. From the following cartesian diagram
\begin{center}
\begin{tikzcd}
\Spec(K) \arrow[r]\arrow[d, "\psi"] & \Spec(A[x]/f(\underline{s},x))\arrow[d, "\varphi"] \\
\Spec(k(\underline{s})) \arrow[r, "\eta"] & \Spec(A)
\end{tikzcd}
\end{center}
one may then deduce a commutative diagram
\begin{center}
\begin{tikzcd}
\pi_1^{\et}(\Spec(k(\underline{s})),\zeta) \arrow[rd, "\alpha"]\arrow[dd, twoheadrightarrow, "\pi_1^{\et}(\eta)"] & \\
& \Sym(Z_{\overline{\eta}}) \\
\pi_1^{\et}(X,\overline{\eta}) \arrow[ur, "\beta"] &
\end{tikzcd}
\end{center}
where $\alpha$ denotes the action of $\pi_1^{\et}(\Spec(k(\underline{s})),\zeta)$ on the fibre of $\psi$ over $\Spec(\overline{k(\underline{s})})$ and $\beta$ the action of $\pi_1^{\et}(X,\overline{\eta})$ on the fibre $Z_{\overline{\eta}}$. Since the $G$-extension $L/k(\underline{s})$ is the Galois closure of~$K/k(\underline{s})$, the image of $\alpha$ is a group that may be identified with $G$. The surjectivity of~$\pi_1^{\et}(\eta)$, which comes from the normality of $X$ \cite[Exposé \uppercase\expandafter{\romannumeral5\relax}, Proposition 8.2]{SGA1}, thus ensures that the image of $\beta$ is also $G$. Thus, there exists a Galois $G$-cover $Y\rightarrow X$ which factors through $\varphi$.

To prove that $\varphi$ is weakly versal, let $E/M$ be a $G$-extension of fields where $M$ contains~$k$. Since $f(\underline{s},x)$ is a generic polynomial for $G$ over $k$, there exists $a\in M^{|\underline{s}|}$ such that $f(a,x)$ is irreducible, $\disc(f).\den(f)$ does not vanish on $a$ and such that $E$ is a splitting field of~$f(a,x)$ over $M$. Hence $a\in X(M)$, and we are to verify that $F\coloneqq a\times_XY$ is $\Spec(E)$. Choose~$\overline{M}$ a separable closure of $E$, so that the morphisms $\Spec(\overline{M})\rightarrow\Spec(E)$ and $\Spec(\overline{M})\rightarrow\Spec(M)$ will both be denoted by $\xi$. Furthermore, we set $\overline{a}\coloneqq a\circ\xi$. We thus have a commutative diagram
\begin{center}
\begin{tikzcd}
\pi_1^{\et}(\Spec(M),\xi) \arrow[r,"\pi_1^{\et}(\underline{a})"]\arrow[dd, "\gamma"] & \pi_1^{\et}(X,\overline{a}) \arrow[r, "\sim"]\arrow[d,twoheadrightarrow] & \pi_1^{\et}(X,\overline{\eta})\arrow[d,twoheadrightarrow] \\
 & G \arrow[r, "\sim"]\arrow[d, hookrightarrow] & G \arrow[d, hookrightarrow] \\
\Sym(\Spec(\overline{M}[x]/f(a,x))) \arrow[r, equals] & \Sym(Z_{\overline{a}}) \arrow[r,"\sim"] & \Sym(Z_{\overline{\eta}})
\end{tikzcd}
\end{center}
where the right part of the diagram is given by the choice of a path between $\overline{a}$ and $\overline{\eta}$ and the left part by functoriality of the \'{e}tale fundamental group. But since the \'{e}tale finite cover $Y\rightarrow X$ corresponds to the morphism $\pi_1^{\et}(X,\overline{\eta})\rightarrow G$, the commutativity of the previous diagram ensures that the $G$-cover $F\rightarrow a$ corresponds to the morphism $\pi_1^{\et}(\Spec(M),\xi)\xrightarrow{\gamma} \Im(\gamma)$. Now, when identifying $\pi_1^{\et}(\Spec(M),\xi)$ with $\Gamma_M$, one may check that the image of $\gamma$ is isomorphic to the group $G$. The Galois extension $E/M$ thus corresponds to the subsequent morphism $\pi_1^{\et}(\Spec(M),\xi)\rightarrow G$, from which one deduces that $F=\Spec(E)$.

\bibliographystyle{myamsalpha}
\bibliography{biblio}

\end{document}